\documentclass[11pt]{amsart}
\usepackage{mathrsfs,amsthm,amsmath,amssymb,bbm,pdfsync,prettyref,graphicx,hyperref}
\pdfoutput=1
\usepackage{fullpage}
\theoremstyle{plain}
\newtheorem{thm}{\protect\theoremname}[subsection]
\newtheorem*{thm*}{\protect\theoremname}

\newtheorem{lem}[thm]{\protect\lemmaname}
\newtheorem*{lem*}{\protect\lemmaname}

\newtheorem{prop}[thm]{\protect\propositionname}
\newtheorem*{prop*}{\protect\propositionname}

\newtheorem{cor}[thm]{\protect\corollaryname}
\newtheorem*{cor*}{\protect\corollaryname}

\newtheorem{fact}[thm]{\protect\factname}
\newtheorem*{fact*}{\protect\factname}
\newtheorem{conj}[thm]{\protect\conjecturename}

\theoremstyle{definition}
\newtheorem{defn}[thm]{\protect\definitionname}
\newtheorem*{defn*}{\protect\definitionname}
\newtheorem{exam}[thm]{\protect\examplename}

\theoremstyle{remark}
\newtheorem{rem}[thm]{\protect\remarkname}

\numberwithin{equation}{subsection}

\newcommand{\cA}{\mathcal{A}}
\newcommand{\cB}{\mathcal{B}}
\newcommand{\cC}{\mathcal{C}}
\newcommand{\cE}{\mathcal{E}}
\newcommand{\cF}{\mathcal{F}}
\newcommand{\cG}{\mathcal{G}}

\newcommand{\cL}{\mathcal{L}}
\newcommand{\cM}{\mathscr{M}}

\newcommand{\cP}{\mathcal{P}}
\newcommand{\cR}{\mathcal{R}}
\newcommand{\cT}{\mathcal{T}}

\newcommand{\R}{\mathbb{R}}

\newcommand{\N}{\mathbb{N}}

\newcommand{\prob}{\mathbb{P}}
\newcommand{\E}{\mathbb{E}}
\newcommand{\eps}{\epsilon}

\newcommand{\indicator}[1]{\mathbbm{1}_{#1}}

\newcommand{\eqdist}{\stackrel{(d)}{=}}
\newcommand{\convdist}{\stackrel{(d)}{\rightarrow}}
\newcommand{\tensor}{\otimes}

\newcommand{\gibbs}[1]{\left\langle#1\right\rangle}

\newcommand{\bsig}{\boldsymbol{\sigma}}
\newcommand{\bfW}{\mathbf{W}}
\newcommand{\bftW}{\tilde{\bfW}}
\newcommand{\talpha}{{\tilde{\alpha}}}
\newcommand{\tbeta}{{\tilde{\beta}}}
\newcommand{\bv}{\mathbf{v}}
\newcommand{\abs}[1]{\lvert#1\rvert}
\newcommand{\norm}[1]{\lvert\lvert#1\rvert\rvert}

\providecommand{\corollaryname}{Corollary}
\providecommand{\definitionname}{Definition}
\providecommand{\factname}{Fact}
\providecommand{\lemmaname}{Lemma}
\providecommand{\propositionname}{Proposition}
\providecommand{\theoremname}{Theorem}
\providecommand{\remarkname}{Remark}
\providecommand{\conjecturename}{Conjecture}
\providecommand{\examplename{Example}}

\newcommand{\documentname}{paper}

\newrefformat{prop}{Proposition \ref{#1}}
\newrefformat{subsub}{Section \ref{#1}}
\newrefformat{fact}{Fact \ref{#1}}
\newrefformat{cor}{Corollary \ref{#1}}
\newrefformat{exam}{Example \ref{#1}}
\newrefformat{rem}{Remark \ref{#1}}
\newrefformat{defn}{Definition \ref{#1}}
\begin{document}

\title{Approximate Ultrametricity for Random Measures and Applications to Spin Glasses}

\author{Aukosh Jagannath}
\date{December 22, 2014}
\keywords{spin glasses, ultrametricity, Ghirlanda-Guerra Identities, mass-partitions}
\subjclass[2010]{ 60G57, 60K35, 82B44, 82D30} 

%\maketitle
\begin{abstract}
In this \documentname, we introduce a notion called ``Approximate Ultrametricity'' which encapsulates
the phenomenology of a sequence of random probability measures  having supports that behave
like ultrametric spaces insofar as they decompose into nested balls.  We provide a sufficient condition for 
a sequence of random probability measures on the unit ball of an infinite dimensional separable Hilbert space to admit such a decomposition, whose
elements we call clusters.  We also characterize the laws of the measures
of the clusters by showing that they converge in law to the weights of a Ruelle Probability Cascade.
These results apply to a large class of classical models in mean field spin glasses. 
We illustrate the notion of approximate ultrametricity by proving two important conjectures 
regarding mixed p-spin glasses.
\end{abstract}
\maketitle
\section{Introduction}

In their study of mean field spin glass models, the authors of \cite{Mez84} predicted 
that the support of the Gibbs measure should be ultrametric in the limit of an infinite number of spins. 
They determined that ultrametricity was a cornerstone for understanding Parisi's Replica Symmetry Breaking
ansatz. Ultrametricity, they  explained, 
accounts for the hierarchical decomposition of the Gibbs measure into ``pure states''
and their ``linear convex combinations'' by allowing one to interpret the latter as  balls, which have a natural
hierarchical structure in ultrametric spaces.  
The physics literature also characterized the laws of these pure states
using the so-called Ruelle Probability Cascades (RPCs) \cite{Der80,Mez84,Rue87}. 
(For a definition of RPCs, see the Appendix.)
In his recent fundamental study, Panchenko has established this ultrametricity property for a natural limiting
object, which he calls the ``Asymptotic Gibbs Measure'', for mean field spin
glass models under a natural condition, the \emph{Ghirlanda-Guerra identities} (see \prettyref{sec:main-results} for a precise definition) \cite{PanchUlt13}. 
This leads to the natural question: 
\begin{quote}
``Is there a sense in which one can see a `pure state' decomposition
occurring at large but finite numbers of spins?'' 
\end{quote}
The answer to this question has proven to be important 
in the physics literature not just for understanding the Replica Theory, 
but also its connection to the TAP approach \cite{MPV87}. This question was first studied by Talagrand in \cite{Tal09} 
where 
he obtained some partial results on this question assuming the  Ghirlanda-Guerra Identities, along with two other conditions, before Panchenko's proof of ultrametricity was known.

We answer this question in the affirmative. To this end, we introduce
the notion of \emph{approximate ultrametricity} to formalize the notion of a sequence of random measures  behaving
``asymptotically ultrametrically'' in that their supports admit a decomposition into ``pure state''-like clusters.  
 We then prove that
a sequence of random probability measures supported on the unit ball of an infinite dimensional separable Hilbert space
is approximately ultrametric provided that this sequence satisfy the Approximate Ghirlanda-Guerra identities along with an additional condition which can be thought of 
as encoding the radii allowed for the clusters. Furthermore, we characterize the laws of the measures of the
clusters by showing that they converge to those of an RPC. 

As a consequence of our studies, one finds that for a large class of mean field spin glass models, the Gibbs measures
do in fact admit such a decomposition at large but finite $N$. 
We discuss this in \prettyref{sec:models}. To further illustrate of our results, 
we prove Talagrand's Orthogonal Structures conjecture and as a consequence verify the 
Dotsenko-Franz-M\'ezard conjecture in a natural regime. We state these results  in \prettyref{sec:OSC-def}.

We state the main definitions and results of this \documentname~in the rest of this section which is organized as follows.
In \prettyref{sec:prelim-def}, we define approximate ultrametricity and the Approximate Ghirlanda-Guerra identities,
along with definitions necessary for understanding the main results of this \documentname.
In \prettyref{sec:main-results}, we state the aforementioned results. We outline the proofs of these results
in \prettyref{sec:outline}. See \prettyref{sec:organization} for an explanation of the organization of the remainder of
this \documentname.

\subsection{The Notion of Approximate Ultrametricity}\label{sec:prelim-def}
We now state the definitions necessary to understand the statements of the main results of this \documentname. 
We begin by introducing the following notions to encapsulate the idea of a sequence of measures behaving
``increasingly ultrametrically'' as discussed above. We begin with the following
 definition. Roughly speaking, this notion should be
thought of as encoding the collections of radii that are allowed
for our ``pure state'' decompositions.

\begin{defn}
Let $\zeta$ be a probability measure on the interval $[-1,1]$. 
A finite increasing sequence of points $\{q_k\}_{k=1}^r$ on the interval $(0,1]$ is said to be $\zeta$-\emph{admissible} if:
\begin{enumerate}
\item They are all continuity points:
\[
\zeta(\{q_k\})=0.
\]
\item There is mass between them:
\[
\zeta([q_k,q_{k+1}])>0.
\]
\item Some but not all of the $\zeta$-mass lies between $0$ and  $q_1$, and $q_r$ and $1$:
\[
0<\zeta([0,q_1])<1\text{ and }0<\zeta([q_r,1])<1.
\]
\end{enumerate}
\end{defn}

With this in hand we can then define what we mean by being approximately
ultrametric. 
This definition encodes the idea that if one picks a (finite) increasing sequence of radii, then there is a sequence
of sets that are hierarchically arranged by inclusion that uniformly ``almost'' exhaust the measure of the space
and for which points within the sets are uniformly close or far depending on the relation of the sets in this hierarchy.
As we will be focusing primarily
on sequences of random measures on the unit ball of an infinite dimensional separable Hilbert space, we will tailor the definition to this setting, though one 
could extend this notion to more general metric spaces. Without loss of generality, we take
this Hilbert space to be $\ell_2$ throughout this \documentname.

In the following for any $\alpha,\beta$ vertices in a tree, 
$\abs{\alpha}$ means the depth of the vertex $\alpha$, $\alpha\precsim\beta$ means that $\alpha$ is along 
the root vertex path to $\beta$, and $\alpha\nsim\beta$ means that $\alpha$ and $\beta$ are on different root-vertex 
paths, in which case we call $\alpha$ and $\beta$ \emph{cousins}. Let $child(\alpha)$ be those $\beta$ such that $\abs{\beta}=\abs{\alpha}+1$ and for which $\alpha\precsim\beta$. For more on this notation see \prettyref{sec:trees}.

\begin{defn}\label{defn:approx-um}
A sequence $\{\mu_N\}_{N=1}^\infty$ 
of random measures supported on the unit ball of $\ell_2$ is said to 
be  \emph{approximately ultrametric} with respect to the measure $\zeta$
 if for every $r$ and every $\zeta$-admissible sequence, 
$\{q_k\}_{k=1}^r$,  there is a sequence of finite rooted trees of depth $r$, $\{\tau_{N,r}\}$, 
and sequences $a_N$, $b_N$, and $\eps_N$ all tending to $0$ such that with probability tending to one, 
there are sets $\{C_{\alpha,N}\}_{\alpha\in\tau_{N,r}}$ with the following properties.
\begin{enumerate}
\item The inclusion-exclusion structure of the sets is natural with respect to the partial order of $\tau_{N,r}$:
\[
C_{\alpha,N} \cap C_{\beta,N} =\begin{cases} C_{\alpha,N} &\text{ if }\beta\precsim\alpha.\\							\emptyset &\text{ if } \beta\nsim\alpha.
\end{cases}
\]
\item The sets almost exhaust the measure at every depth: for each $k\in[r]$, 
\[
\sum_{\abs{\alpha}=k} \mu_N(C_{\alpha,N}) \geq 1-\eps_N.
\]

\item For $\alpha\in\tau_{N,r}\setminus\partial\tau_{N,r}$ 
\[
0\leq \mu_N(C_{\alpha,N}) -\sum_{\beta\in child(\alpha)} \mu_N(C_{\beta,N}) \leq \eps_N.
\]
\item Points in the same cluster are uniformly close with high probability in the environment: for every $\alpha$ in $\tau_{N,r}$, 
\begin{equation}\label{eq:unif-close}
\mu_N \left(\sigma^1,\sigma^2\in C_{\alpha,N} :(\sigma^1,\sigma^2)\leq q_{\abs{\alpha}}-a_N\right)\leq b_N
\end{equation}
\item Points in clusters that are cousins are uniformly far apart: for every $\alpha\nsim\beta$ in $\tau_{N,r}$, if $\gamma\prec\alpha$  and $\eta\prec\beta$ are such that
$\abs{\gamma}=\abs{\eta}=\abs{\alpha\wedge\beta}+1$, then
\begin{equation}\label{eq:unif-far}
\mu_N\left(\sigma^1\in C_\gamma,\sigma^2\in C_\eta :(\sigma^1,\sigma^2) \geq q_{\abs{\gamma\wedge\eta}+1} +a_N\right) \leq b_N.
\end{equation}
\end{enumerate}
We say that this sequence is \emph{regularly} approximately ultrametric if, furthermore, there is a sequence $m_N
\rightarrow \infty$ such that $\tau_{N,r}$ is the $m_N$-regular tree of depth $r$. We call these $C_{\alpha,N}$ the  
\emph{clusters} corresponding to the sequence $\{q_k\}$.
\end{defn}

\begin{rem}
We would like to point out here that as stated the above
definition uses very little about $\zeta$. All we needed was a notion of admissible
sequences. This notion is invariant under changing the measure to an equivalent measure.
If another measure $\nu$ is equivalent to $\zeta$ in the sense
that there is a positive function $f$ with $d\nu =f d\zeta$, then
the class of admissible sequences  for $\nu$ and $\zeta$ are the same. The
reader is encouraged, however, to think of 
$\zeta$ as the limit of the mean law of $(\sigma^1_N,\sigma^2_N)$ where $\sigma_N^1$ and $\sigma_N^2$ are 
drawn i.i.d from $\mu_N$. 
The reason for this will become clear in the next section.
\end{rem}

\noindent  The first three properties together are called $(\eps_N,0)$-\emph{hierarchical exhaustion}, and the last two 
properties are called $(a_N,b_N)$-\emph{hierarchical clustering}. For a more thorough explanation and motivation of the 
above definitions,  see  \prettyref{sec:um}.
Also see \prettyref{fig:h-exhaust}. 
 
\begin{figure}[t]\label{fig:h-exhaust}
  \centering
    %\reflectbox{
      \includegraphics[width=0.4\textwidth,height=0.3\textheight]{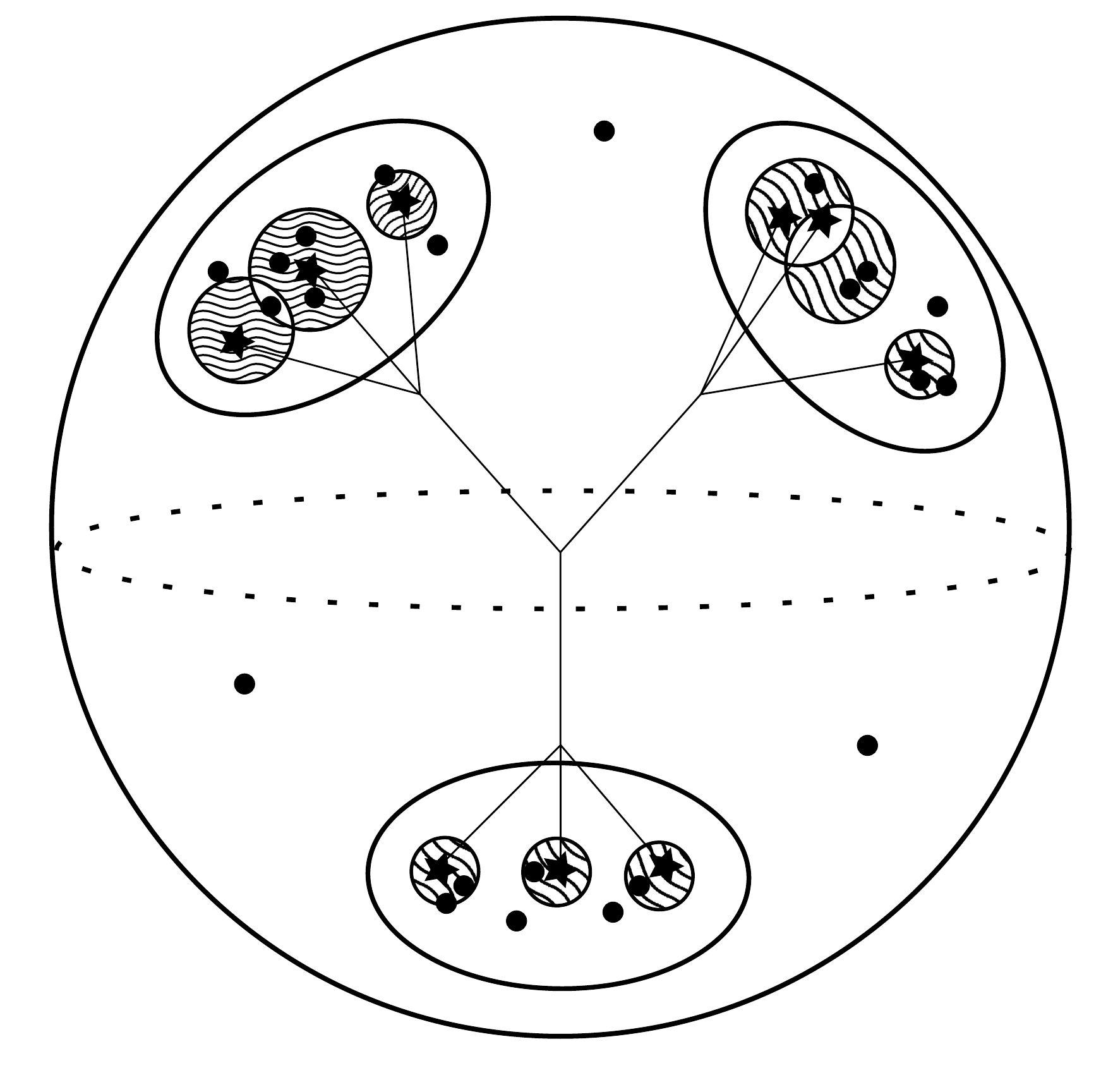}%}
  \caption{A collection of clusters that satisfy \prettyref{defn:approx-um}. The centers of the balls are denoted by stars,
  the sets are denoted by the discs with wavy lines, and the measure puts even mass at all of the dots and stars. In the top left, one of the discs has a center that is too close the center of the other, so it is excised from the exhaustion.}
  
\end{figure}

These clusters are to be compared with the ``pure states'' of physicists \cite{MPV87}. 
In particular, we think of them as an approximation to the pure states
at a large but finite number of spins. %(c.f. the ``approximate pure states'' in \cite{PanchHEpure2013}). 
As we shall see shortly, these clusters are constructed
such that they are not only  ``pure state''-like in the sense of radii, but they are in fact
approximately balls, this approximation becoming exact in the thermodynamic limit.

\subsection{Statement of main results}\label{sec:main-results}
In this \documentname, we obtain sufficient conditions to conclude that a sequence of random measures
on the unit ball of $\ell_2$ is approximately ultrametric with respect to
some $\zeta$. In the spin glass literature,
a natural symmetry property of such measures  has been identified as being related to ultrametricity, which we define
presently. Before we do so, we need a few technical definitions.

Let $\mu$ be such a random measure on a fixed probability space 
$(\Omega,\cF,\prob)$. Draw $(\sigma^i)_{i=1}^\infty$ i.i.d. from $\mu$, that is, $\mu^{\tensor\infty}$ is a
regular conditional probability distribution for $(\sigma^i(\omega))$ (see \cite{AldExch83}). 
Let $R_{ij}=(\sigma^i,\sigma^j)$ be the inner
product of these random variables. Here and in the following, the subscripts always refer
to the index of the element of the draw.
Finally let $R^n=(R_{ij})_{i,j\in[n]}$ be the $n$-by-$n$ array of pairwise inner products of the first $n$ draws.

With this in hand, we can then define the symmetry property.
\begin{defn}
We say that a sequence $\{\mu_N\}$ of random measures supported on the unit ball of $\ell_2$
satisfies the \emph{Approximate Ghirlanda-Guerra Identities} (AGGIs) 
if  for every $n\in\N$,   bounded Borel measurable $f$ on $[-1,1]^{n^2}$, and  $\psi\in C[-1,1]$,  
\begin{equation}
\lim_{N\rightarrow\infty}\abs{n\E\gibbs{ f(R^n)\psi(R_{1,n+1})}_N - \E\gibbs{f(R^n)}_N\E\gibbs{\psi(R_{12})}_N -\sum_{k=2}^n \E\gibbs{f(R^n)\psi(R_{1k})}_N}=0,
\end{equation}
where $\gibbs{\cdot}_N$ denotes the expectation with respect to $\mu_N$, and $R_{ij}$ and $R^n$ are as above. 
\end{defn} 

With this language we can then state the main result of the \documentname.

\begin{thm}\label{thm:main-result-1}
Let $\{\mu_N\}$ be a sequence of random probability measures supported on the unit ball of $\ell_2$ 
on a fixed probability space $(\Omega,\cF,\prob)$. Suppose that this sequence satisfies the Approximate 
Ghirlanda-Guerra identities and that
\[
\zeta_N = \E\mu_N^{\tensor2}(R_{12}\in \cdot)\rightarrow\zeta
\]
weakly for some $\zeta$. 
Then this sequence is regularly approximately ultrametric with respect to $\zeta$. 
\end{thm}

We are also able to characterize the laws of the sequence of weights $\{\mu_N(C_{\alpha,N})\}$. 
To do this we introduce the following two definitions. Let $\cA_r$ be the rooted tree of depth $r$ for 
which each non-leaf vertex has $\N$ children. For readers unfamiliar with $\cA_r$, see \prettyref{sec:trees}.

\begin{defn}
A collection of random variables $(V_\alpha)_{\alpha\in\cA_r}$ is said to be
in \emph{standard order} if the following is true. The $V_n$ (i.e. the weights of the vertices
at the first level) are arranged in decreasing order. The weights corresponding
to the children of any vertex $\alpha$, 
$V_{\alpha n}$, are arranged in decreasing order.
\end{defn}

Let $\mu_N$ and $\zeta$ satisfy the conditions of \prettyref{thm:main-result-1}. Then for any $\zeta$-admissible
$\{q_k\}_{k=1}^r$, let  $C_{\alpha,N}$ be the corresponding clusters, and set 
\[
\tilde{Y}_{\alpha,N}=\begin{cases} C_{\alpha,N} & \text{ for } \alpha\in\tau_{N,r}\\
 0 &\text{ otherwise}
\end{cases}
\]
where $\tau_{N,r}$ is the $m_N$-regular tree from \prettyref{defn:approx-um}. Finally, let $(Y_{\alpha,N})_{\alpha\in\cA_r}$ be $(\mu_N(\tilde{Y}_{\alpha,N}))_{\alpha\in\cA_r}$ arranged in standard order. We then have the following theorem.

\begin{thm}\label{thm:main-result-2}
Let $(Y_\alpha)$ be distributed like the weights of a Ruelle Probability Cascade with parameters 
$\zeta_k=\zeta[q_k,q_{k+1}]$ with $q_0=0$ and $q_{r+1}=1$. Then
\[
(Y_{\alpha,N})\convdist (Y_\alpha).
\]
 
\end{thm}
This convergence in distribution is topologized in \prettyref{sec:cascades}. For the reader's convenience, basic facts about Ruelle Probability Cascades are reviewed in the appendix. 
Alternatively, see \cite{PanchSKBook}. 

We would like to end this section with a remark regarding the
possibility of quantifying these results. At first glance
one might expect that the above results are  
unquantifiable. This, however, is not the case. In \prettyref{sec:quant-version}, we demonstrate that if one can 
obtain a uniform rate of convergence of the probabilities
of a particular class of sets (having to do with $R^n$), 
then one can obtain an estimate on the rates of convergence of the above. In particular, fix an admissible sequence 
$\{q_k\}_{k=1}^r$ and let $\zeta_1=\zeta[0,q_1)$, and 
 let $\mathscr{B}$ be defined by 
\[
\mathscr{B}=\{ E:\exists k: E=\{ x_1 \geq q_k + \eps, x_2 \geq q_k, x_3 <q_k\}, \text{ or }, \exists k: E=\{x_1\leq q_{k}-\epsilon,x_{2}\geq q_k,x_{3}\geq q_{k}\}\}.
\]
These sets measure the failure of ultrametricity. 
Furthermore, let $\mathscr{A}(\{q_k\})$ be defined as in \prettyref{sec:ballweights}. Roughly speaking, these sets
have to do with the probabilities of balls of certain radii. 
Finally let $D_1(N;\{q_k\})$ and $D_2(N;\{q_k\})$  be such that
\[
D_1(N)=\sup_{A\in\mathscr{A}}\abs{P_N(A)-P(A)}
\text{ and } 
D_2(N)=\sup_{B\in\mathscr{B}}\abs{P_N(B)-P(B)}.
\]
If we assume that there is a monotone decreasing function $D(N)$
that goes to $0$ as $N\rightarrow\infty$ such that
\[
D(N)\geq D_1(N) \vee D_2(N),
\]
then we have the following (probably highly sub-optimal) result.
\begin{thm}\label{thm:quant-clust}
Let $\{q_k\}_{k=1}^r$ be an admissible sequence with parameters
$\zeta_k=\zeta[0,q_k)$. Then there are functions $n_0(N)$,
$\nu(N)$, and $m_{**}(N)$ such that $\mu_N$ admits a collection
$\{C_{\alpha,N}\}_{\alpha\in\tau_{m_{**}(N)}}$ that
is $(2^{-\nu+1},0)$-hierarchically exhausting and 
$(1/2^{n_0},1/2^{n_0/2})$-hierarchically clustering with probability 
greater than $1-1/2^{\nu(N)}$.
Furthermore we have the bounds,
\begin{align*}
\nu (N)&\geq \Omega(\log\log\log\log(1/D(N)))\\
n_0 (N) &\geq \Omega(\log\log\log\log(1/D(N)))\\
m_{**}(N) &\geq \Omega((\log\log\log(1/D(N)))^c),
\end{align*}
where these inequalities are to be understood 
up to constants that depend on $r$
and $\zeta_1$, and $c$ also depends on these parameters.
\end{thm}
 We would like to re-iterate that the rates we get are to be viewed as 
 most likely \emph{highly} 
sub-optimal. 
We do not try to find the optimal rates
as we believe the real issue will be finding $D(N)$. To see why, note that 
to obtain a uniform rate for the class $\mathscr{B}$ is equivalent to proving the so-called
``Strong Ultrametricity'' conjecture \cite{TalBK11}. Note also that one might instead
attempt to obtain a rate of convergence for the limiting overlap distribution as a whole. 
This should allow one to simultaneously obtain rates for the classes $\mathscr{A}$ and
$\mathscr{B}$. 

A final remark: G. Parisi mentioned in private communication that simulations suggest  
 that the free energy decays like a power law,
so it seems reasonable to assume 
$D(N)$ decays like a power law.

%\addcontentsline{toc}{section}{Acknowledgements}
\section*{Acknowledgments}
The author would like to thank Dmitry Panchenko for lending his insight at the early stages 
of this project, as well as many helpful correspondences.
The author thanks Antonio Auffinger for many helpful
discussions and for a very careful reading of an early version of this paper. 
The author also thanks his Ph.D. advisor G\'erard Ben Arous for his support. 
This research was conducted while
the author was supported by an NSF Graduate Research Fellowship DGE-0813964 
and NSF Grant DMS 1209165.

\section{Applications to Spin Glasses}\label{sec:applications}
In this section we describe how to apply the above results to spin glasses.
We begin in \prettyref{sec:models}
with a discussion of how to apply these results to  specific 
models on the hypercube.
We then use the results to recover a similar result of Talagrand in \prettyref{sec:tal-pure-intro}.
We end by discussing how to use these results to prove 
two important conjectures in spin glasses in \prettyref{sec:OSC-def}.

\subsection{Models on the Hypercube}\label{sec:models}
In this section we explain how to apply the above results to spin glass
models on the hypercube. We begin by explaining how to view spin
glass models in this framework and then discuss specific 
models to which the results apply.

Let $\Sigma_N=\{-1,+1\}^N$ 
be the $N$-dimensional hypercube.
We focus on Gaussian spin glass models which have the
following form. For each $N$, we have a Hamiltonian 
$(H_N(\sigma))_{\sigma \in \Sigma_N}$ which
is a centered Gaussian process indexed by $\Sigma_N$ with covariance
\begin{equation}\label{eq:cov-str}
\E H_N(\sigma^1)H_N(\sigma^2) = N\xi(R(\sigma^1,\sigma^2))
\end{equation}
where 
\[
R(\sigma^1,\sigma^2)=\frac{1}{N} \sum_{i=1}^N\sigma^1_i\sigma^2_i
\]
is called the \emph{overlap}. Corresponding to this sequence of Hamiltonians, is the
sequence of Gibbs measures
\[
G_N(\sigma)=\frac{e^{-\beta H_N(\sigma)}}{Z_N}
\]
where the normalization $Z_N$
is the \emph{partition function}.

To place this in the above framework,  
view $\Sigma_N$ as included in $\ell_2$ through the natural inclusion map
$i:\Sigma_N\hookrightarrow \ell_2$ defined by
\[
(\sigma_1,\ldots,\sigma_N)\stackrel{i}{\mapsto} 
(\frac{1}{\sqrt{N}} \sigma_1,\ldots,\frac{1}{\sqrt{N}}\sigma_N,0,\ldots).
\]
Notice that under this inclusion, the overlap becomes the $\ell_2$ inner product,
\[
R(\sigma^1,\sigma^2)=(i(\sigma^1),i(\sigma^2))_{\ell_2}.
\]
Furthermore we see that $\norm{i(\sigma)}_2 =1$, so that 
$i(\Sigma_N)\subset B_{\ell_2}(0,1).$
The push-forward of the Gibbs measures through this map
$\mu_N = i_* G_N$
is the desired sequence of random measures on the unit ball that are a.s.
supported on the unit sphere $\{\norm{\sigma}=1\}$. 

The arrays $R^n$ from \prettyref{sec:prelim-def} are
then the leading principal minors of the overlap array since
\[
\E G_N^{\tensor\infty}((R(\sigma^i,\sigma^j))_{i,j\in[n]}\in \cdot )
=\E\mu_N^{\tensor\infty}(R^n\in\cdot)
\]
and similarly $\zeta_N$ is the overlap distribution
corresponding to $G_N$
\[
\zeta_N(\cdot)=\E G_N^{\tensor 2}(R(\sigma^1,\sigma^2)\in \cdot).
\]

To understand these results, suppose for a moment 
that we know that  $\{\mu_N\}$ satisfies the AGGI's and 
$\zeta_N\rightarrow\zeta$ for some $\zeta$ with 
a non-empty collection of admissible sequences. 
Then \prettyref{thm:main-result-1} tells us that for any $\zeta$-
admissible sequence $\{q_k\}_{k=1}^r$,  with high probability 
in the choice of $H_N$, there is a nontrivial decomposition of the hypercube
into clusters  which organize by their overlap structure off of which
the measure places essentially no mass, 
as with the pure states described in \cite{MPV87}. Furthermore,
\prettyref{thm:main-result-2} says that these clusters have masses that approach that of a Ruelle Probability Cascade
as with the pure states from the physics literature. Thus
we can see the Ruelle Cascade arising at finite $N$ by looking at 
where the measure places mass.

\begin{rem}
One would like to think of these sets as balls on the 
hypercube in the induced metric from $\ell_2$. As is evident from the proof
of \prettyref{thm:main-result-1}, this is quantitatively not far 
from what is proven.
\end{rem}

We now discuss specific models to which our results apply.

\begin{exam}[Mixed p-spin glasses]
The \emph{mixed p-spin glass} model is the model with Hamiltonian $H_N$
of the above form with covariance structure defined by \prettyref{eq:cov-str}
where $\xi$ is of the form
\[
\xi(t)=\sum_{p\geq 1} \beta_p^2 t^p
\]
and is such that $\xi(1+\eps)<\infty$ for some $\eps>0$. A mixed p-spin
glass model is said to be \emph{generic} if the monomials
$\left\{x^p:\beta_p >0\right\}$ 
are total in $(C[0,1],\norm{\cdot}_\infty)$ (see the M\"untz-Szaz theorem \cite{DymMckean} for a necessary and sufficient condition). 
It is well-known \cite{PanchSKBook} that for generic models, the corresponding
sequence of $\mu_N$ satisfy the AGGI's. Furthermore, it is known
\cite{TalPM06} that the sequence $\zeta_N$ has a unique limit point, 
$\zeta$. It is also known \cite{AufChen13} that at low-temperature,
that is, $\beta$ large, $\zeta$ has admissible sequences of length at
least 2. Our results thus apply to generic models at low temperature.
\end{exam}

\begin{rem}
It is not known whether or not the Ghirlanda-Guerra identities 
hold for a larger class of models. Indeed
it can be shown that the Sherrington-Kirkpatrick (SK) model ($p=2$) with no external field does not satisfy the 
Ghirlanda-Guerra identities as it violates the Talagrand Positivity
Principle (See \prettyref{sec:GGI} for more). 
That being said,  one can
prove that for any mixed $p$ model there is a perturbation of the Hamiltonian that does not change the free energy such 
that the perturbed model has the Ghirlanda-Guerra Identities  \cite{TalBK11,PanchSKBook}. It is not known if the overlap distribution converges in
a more general setting.
\end{rem}
\begin{exam}[REM]
The Random Energy Model (REM) \cite{Der80} is the model with covariance structure of the form \prettyref{eq:cov-str}
where $\xi$ is given by 
\[
\xi(t) = \indicator{t=1}.
\]
It is known that for the REM, the measures $\mu_N$ satisfy the AGGIs \cite{BovKurk04-1}
and that the $\zeta_N$ weakly converge to
\[
\zeta=(1-\frac{\beta}{\beta_c})\delta_0 + \frac{\beta}{\beta_c}\delta_1
\]
for $\beta>\beta_c$ where $\beta_c = \sqrt{2\log2}$ \cite{Bov12}. Thus our
results apply to the REM. In particular they give a decomposition with
precisely one level (here the only admissible sequences consist of a single
point $q\in(0,1)$).
It has been shown that at low temperature,
the Gibbs measure converges as a point process  to a Ruelle Cascade \cite{Bov12}. 
Our results give a new proof of this, as well as showing how the cascade 
can be seen to be occurring qualitatively at finite $N$.
\end{exam}

\begin{exam}[GREM]
The Generalized Random Energy model (GREM) \cite{Der85} has a slightly non-standard covariance structure. 
Fix an $r\in\N$, and sequences
\[
0=q_0 <q_1<\ldots<q_{r}\leq 1
\text{ and }
0=\zeta_{-1}<\ldots<\zeta_{r+1}=1.
\]
Let $\log(\alpha^N_i) = N(q_{i+1}-q_{i})\log(2)$ and let
\[
\xi(\sigma^1,\sigma^2) = \sum_{k\leq r} (\zeta_{k+1}-\zeta_{k}) \indicator{(\sigma^1(k),\sigma^2(k))=1},
\]
where $\sigma(k)$ is the vector with the same first $N\sum{i\leq k}\lg(\alpha_i^N)$ components as $\sigma$
and has the remaining coordinates set to $0$. The GREM
is the model whose covariance is of the form 
\[
\E H_N(\sigma^1)H_N(\sigma^2) = N\xi(\sigma^1,\sigma^2).
\]
It is known that the sequence $\zeta_N$ converges to $\zeta$ 
such that
\[
\zeta\{q_k\} = \zeta_{k+1}-\zeta_{k}
\]
at sufficiently low temperature. (See \cite[Chapter 10]{Bov12} for the 
relevant temperature ranges.) 
It is also known that the sequence of $\mu_N$ corresponding to the GREM 
has the Ghirlanda-Guerra property \cite{Bov12,BovKurk04-1,BovKurk04-2}. Our results then 
apply to the GREM at low-temperature. In particular, they tell us
that at sufficiently low temperature the admissible sequences can be taken
to be any sequence that interlaces the above $q_k$. Taking a sequence
that is arbitrarily close to the $q_k$ then shows us how the Ruelle Cascade
predicted in \cite{Der85} and \cite{Rue87} appears at finite $N$ as
we get clusters of sets that are effectively balls on which most of
the mass is supported that are indexed by the $m_N$-regular rooted
tree of depth $r$ whose masses are close, in law, to those of an RPC.
In particular, as with the REM, 
our results provide a new proof of convergence 
to an RPC, as well as providing a sense in which this structure
arises at finite $N$ at low temperature.
\end{exam}
There are many other models which are expected to satisfy the AGGIs. In particular, it is known after using
a perturbation as with the mixed p models for the Diluted SK model \cite{PanchSGSD13}, the Random K-SAT model
 \cite{PanchHEpure2013}, and the Edwards-Anderson model \cite{ContStarr13} to name a few. 
To our knowledge, it is not yet known in any of these models if the overlap distribution converges.

\subsection{A new proof of Talagrand's pure state construction}\label{sec:tal-pure-intro}
We now discuss how to use our results to recover a result of Talagrand
on the existence of approximate pure states. 
Until now, the best answer to the question of the existence of approximate pure states at finite particle number was 
\cite[\theoremname~2.4]{Tal09}. Talagrand showed that under certain
conditions on the overlap distribution
there are sets inside the support of the
Gibbs measure such within any given set, the 
overlap between two configurations is almost $q_*$, the supremum of the 
support of the overlap distribution, 
on average. Furthermore, these sets are disjoint,  exhaust the measure in 
the limit, and 
have masses which converge in distribution to a Poisson-Dirichlet process.

The clusters from the above should be compared with these sets.
As a consequence of the proofs of \prettyref{thm:main-result-1} and 
\prettyref{thm:main-result-2}, 
with minor modifications, one 
can recover the existence of these sets up to a small correction. 
Furthermore, if there are admissible sequences of length at least
two,  one gains
additional information about the distance between the sets in the 
sense of \prettyref{thm:main-result-1} and
\prettyref{thm:main-result-2}. We state this  as a corollary. 

\begin{cor} \label{cor:tal-pure}
Let $\mu_N$ be as above. Assume that $\zeta$ is
supported on $[0,q_*]$ and that $\zeta\{q_*\}>0$. 
Then there is a sequence of random sets $A_{k,N}\subset\Sigma_N$ 
such that $\mu_N (A_{k,N})$ converge in distribution
to the points of a $PD(1-\zeta(\{q_*\}))$ and such that for any $\eps$ 
positive, there 
exists an $N_0$ such that for $N>N_0$,
\begin{equation}\label{eq:tal-pure}
\int_{A_{k,N}\times A_{k,N}} \abs{R_{12}-q_*} d\mu_N^{\tensor2} < \eps 
\mu_N(A_{k,N})^2 +o(1)
\end{equation}
for all $k$, with probability at least $1-\eps$.  Furthermore, for all 
$\eps$ positive and $k_0\in\N$, there is an
$N_0$ such that for $N>N_0$ the expression \prettyref{eq:tal-pure} holds 
without the $o(1)$ correction for those
$k\leq k_0$. Finally, these $A_{k,N}$ are the mass-rearranged leaves of a 
random sequence of sets $A^\prime_{\alpha,N}$ whose masses converge to an 
RPC with parameters $\zeta_k = \zeta[q_k,q_{k+1})$ 
where $q_r=q_*$.
\end{cor}
\noindent Here $PD(\theta)$ is the Poisson-Dirichlet Process 
\cite{PanchSKBook}. In terms of the two-parameter PD process, this is $PD(\theta,0)$. 

This is to be compared with the result of Talagrand where he obtains sets as above where
\prettyref{eq:tal-pure} holds without the $o(1)$ correction for all $k$. In exchange for price of the $o(1)$
correction, however, we get additional structure. If $\{q_k\}_{k=1}^r$ is $\zeta$-admissible and $r$
is at least two, 
then these sets are the leaves of a collection of sets $\{A_{\alpha,N}\}_{\cA_{r+1}}$ that hierarchically cluster
and the masses of this latter sequence of sets converges to an RPC. 

\subsection{Talagrand's Orthogonal Structures Conjecture and The Dotsenko-Franz-M\'ezard Conjecture}\label{sec:OSC-def}

We now discuss how to use the above techniques to prove two conjectures regarding 
mixed p-spin glass models on the hypercube. In \cite{Tal07}, Talagrand conjectured that 
the support of the Gibbs measure at low temperature admits a special decomposition which he called 
an ``Orthogonal Structure''

\begin{defn}
A spin glass model on the hypercube
with corresponding Gibbs measures $G_N$  is said to admit an \emph{Orthogonal Structure}
if there is a sequence $(a_k)_{k\geq0}$ with $a_k>0$ such that  for any $k_0\in\N$ and $\eps$ positive,  there is an $N_0$ such that for $N\geq N_0$,
with probability at least 3/4, there is a random collection of sets $\{A_{k,N}\}_{k\leq k_0}\subset\Sigma_N$ such that 
\[
G_N(A_{k,N})\geq a_k
\] 
and on these sets, the points in different sets are almost orthogonal in an $L_1$ sense: for each $k,l  \leq k_0$ with $k\neq l$
\[
\gibbs{\abs{R_{12}}\indicator{\sigma^1\in A_{k,N},\sigma^2\in A_{l,N}}}<\eps.
\]
\end{defn}
He then conjectured that  at zero external field, mixed p-spin glass models admit orthogonal structures. 

\begin{conj} [Orthogonal Structures Conjecture]
The Gibbs measures for a  mixed p-spin glass model with $0$ external field admit an Orthogonal Structure.
\end{conj}

As an illustration of the use of approximate ultrametricity, we prove this  conjecture for generic models.
\begin{thm}[Orthogonal Structures Conjecture]  
For a generic mixed p-spin glass model,  $\zeta(\{0\})>0$ if and only if 
Gibbs measures $G_N$ admit an Orthogonal Structure.
\end{thm}
\begin{rem}
In fact our proof shows that any sequence $\{\mu_N\}$ that satisfies the assumptions of \prettyref{thm:main-result-1}
and has $\zeta(\{0\})>0$ admits an Orthogonal structure except with the random sets as subsets of the unit ball
of $\ell_2$ rather than the hypercube.
\end{rem}

We prove this result in the ``Replica Symmetry Breaking''
regime, that is $\zeta(\{0\})<1$, which follows from Approximate ultrametricity of the sequence 
$\mu_N$ with respect to $\zeta$.  This is proven in \prettyref{thm:low-temp-OSC}. 
That the result holds in the ``Replica Symmetric'' regime $\zeta(\{0\})=1$,  and that an atom at zero is necessary
both immediately follow Markov's inequality and weak convergence, so their proofs are omitted.

In \cite{Tal07},  Talagrand explored a conjecture of Dotsenko, Franz, and M\'ezard, that for the Sherrington-Kirkpatrick
model without external field, for all negative exponents $a$, 
\[
 \lim_{N\rightarrow\infty} \frac{1}{aN}\E \log Z^a_N = \lim_{N\rightarrow\infty}\frac{1}{N}\E \log Z_N 
\]
where $Z_N$ is the partition function \cite{DFM94}. Talagrand showed that this conjecture holds for generic mixed p-spin glass models
provided the Orthogonal Structures conjecture holds. Using Talagrand's result and the above, we  find
that 
the Dotsenko-Franz-M\'ezard conjecture holds in this regime for generic mixed p-spin glass
models.
\begin{cor}[Dotsenko-Franz-M\'ezard conjecture]
Suppose $Z_N$ is the partition function for a generic mixed p-spin glass Hamiltonian such that $\zeta(\{0\})>0$, then
for all $a$ negative
\[
\lim_{N\rightarrow\infty}\frac{1}{aN}\log\E Z^a_N = \lim_{N\rightarrow\infty} \frac{1}{N}\E\log Z_N.
\]
\end{cor}

\subsection{Organization of Document}\label{sec:organization}
The remainder of this document is organized as follows. In \prettyref{sec:prelims}, we introduce
the necessary preliminaries to understand the outline of the proofs of \prettyref{thm:main-result-1} and
\prettyref{thm:main-result-2}. We outline these proofs in \prettyref{sec:outline}.
In Sections \ref{sec:ballweights}-\ref{sec:cascades} we prove these results. In \prettyref{sec:OSC} we prove
the Orthogonal Structures conjecture in the Replica Symmetry Breaking regime. In \prettyref{sec:tal-pure}
we prove \prettyref{cor:tal-pure}. In \prettyref{sec:quant-version} we
discuss quantifying the above results. We end the document with the Appendix which contains miscellaneous results
of relevance to the paper. 

\section{Preliminaries}\label{sec:prelims}
\subsection{Some notation regarding trees}\label{sec:trees}

In the following, we will be frequently working with rooted trees. 
We designate the root by $\emptyset$.  
$p(\alpha)$ denotes is the root-vertex path of $\alpha$. We denote the least common
ancestor of $\alpha$ and $\beta$ by $\alpha\wedge\beta$.  For a tree $\tau$, we let $\partial\tau$ denote its leaves and 
let $\abs{\tau}$ denote its cardinality.

In the subsequent, we work with rooted trees of a particular form.
Let $\{E_k\}_{k=1}^r$ be a collection of subsets of $\N$ and let  $E_0=\{\emptyset\}$.
We think of 
\[
\bigcup_{k\leq r}E_1\times\ldots\times E_k
\]
as a tree as follows. The vertices at depth $k$  correspond to $k$-tuples in
$E_1\times\ldots\times E_k$, where the coordinates sequentially describe the path from the root 
to the vertex (omitting the root). We will be most interested in the case $E_k=\N$ for all $k$. 
We denote this space by $\cA_r=\cup_{k\leq r}\N^k$. As an example of this notation, take the tree
$\cA_3$. Then $\alpha=(1,2)$ corresponds to the second child of the first child of the root.
We define  \emph{shaped} trees.

\begin{defn}
Fix $r\in\N$ and let $(m_1,\ldots,m_r)\in\N^r$. The $(m_1,\ldots,m_r)$-\emph{shaped tree} 
is the tree 
\[
\tau =\bigcup_{k\leq r}[m_1]\times\ldots\times[m_k].
\]
we denote the shape of a tree by $shape(\tau)$. Let $\cT_r$ denote the space of all $(m_1,\ldots,m_r)$ shaped
trees for all $(m_1,\ldots,m_r)\in\N^r$ (in particular, $m_i<\infty$). We call $\cT_r$ the space of finitely
shaped trees of depth $r$.
\end{defn}

We think of $\tau$ with shape $(m_1,\ldots,m_r)$ as included in $\cA_r$ through the natural set
inclusion. Similarly define $\tau^i$ to be the (shifted) tree obtained by adding $i\cdot m_1$ to the first coordinate
of every (non-root) vertex in $\tau$, where $\tau^1$ is just $\tau$ as above.
We also view $\tau^i$ as included in $\cA_r$ in the natural way.
 By $\tau_{m}$ we always mean the $(m,\ldots,m)$-shaped tree.

Once we have tree shapes we also want to study a pruning of the infinite tree
$\cA_r$ into the shape defined by a $\tau\in\cT_r$.

\begin{defn}
Fix a $\tau\in\cT_r$ with shape $(m_1,\ldots,m_r)$. A \emph{$\tau$-pruning of $\cA_{r}$} 
is the (infinite) subtree we get by organizing the children at level 1 into groups of cardinality $m_{1}$ (here
we have not modified the tree), and to each of those children we only
consider the first $m_{2}$ children. At level $k \geq 2$ keep only
the first $m_{k}$ children. 
\end{defn}

\subsection{Some properties of Ultrametric Spaces} \label{sec:um}
For the following work we need to abstract properties of ultrametric spaces. We will focus
on two properties of collections of balls, one is  measure theoretic and the other is geometric measure theoretic. 

Recall  that balls in ultrametric spaces have special inclusion properties. If we fix  two 
balls of radius at most $r$, then either they are disjoint or one is contained in 
the other. In particular a decreasing sequence of radii corresponds
to a partition of space into balls that are hierarchically arranged. 
To be precise, one can index these balls by a rooted tree in such a way
that the balls of parents contains those of their children and balls of cousins are disjoint.

If a metric space is ``almost'' ultrametric, we might expect to have balls
in the support that ``almost'' behave this way. We might expect to find
a sequence of balls that ``almost'' exhaust the measure and are uniformly ``almost'' disjoint.
To make this idea precise we introduce the idea of \emph{hierarchical exhaustions}. 
\begin{defn}
Fix $r$ and a $\tau\in\cT_r$. A collection of sets $\{B_\alpha\}_{\alpha\in \tau}$ is said to be
an \emph{$(\eps,\delta)$-hierarchical exhaustion} of $\mu\in \Pr(B_{\ell_2}(0,1))$ if:
\begin{enumerate}

\item The sets are hierarchically arranged by inclusion:
\[
B_\alpha \subset B_\beta \text{ if }\beta\precsim\alpha.
\]

\item The sets corresponding to cousins have uniformly small intersections:
\[
\sum_{\substack{\alpha,\beta\in\tau\\ \alpha\nsim\beta}} \mu (B_\alpha\cap B_\beta) \leq \delta.
\]

\item The sets almost exhaust the measure at every depth: for each $k\in[r]$, 
\[
\sum_{\abs{\alpha}=k} \mu(B_\alpha) \geq 1-\eps.
\]

\item The sets corresponding to children exhaust the sets corresponding to parents: for $\alpha\in\tau\setminus\partial\tau$ 
\[
\mu(B_\alpha) -\sum_{\beta\in child(\alpha)} \mu(B_\beta) \in [0,\eps).
\]
\end{enumerate}
\end{defn}

\begin{rem}
We would like to point out that there is no implication between (3) and (4). Observe
that (3) does not imply (4) as it does not provide control on the intersections. By
the same token, (4) does not imply (3) as it provides no control on the number of children.
\end{rem}
\noindent We invite the reader to compare this with (1)-(3) in \prettyref{defn:approx-um}.

Another consequence of ultrametricity is that if we take two points in the same ball
of radius $r$ then their distance is at most $r$. Similarly, if we take two balls, $B_1$ and $B_2$ of
radius $r$ whose centers are $r+\eps$-separated and two points 
$x_1\in B_1$ and $x_2\in B_2$, then $x_1$ and $x_2$ are at least $r+\eps$ separated. 
This is to be contrasted with the setting of regular metric spaces where both inequalities
are off by an additive factor of $r$ and $-r$ respectively. 

Again for an ``almost ultrametric'' space, we may not be able to see such a precise structure
as it may happen that the sets seen above have centers that are too
close or  have non-trivial intersections. Instead we might hope 
that the distances between points as above at least behave like the 
distances to the centers of the balls with high probability. 
To make this precise we introduce the notion of \emph{hierarchical clustering}
\begin{defn}
Fix $r$, a $\tau\in \cT_r$, and $\{q_k\}_{k=1}^r$. A collection of sets $\{C_\alpha\}_{\alpha}$ is said to be 
an $(\eps,\delta)$-\emph{hierarchical clustering} for  a measure $\mu$ with respect to the sequence $\{r_k\}$ if 
\begin{enumerate}
\item Points are uniformly close within clusters: for every $\alpha$ in $\tau_{m_N}$,
\begin{equation}\label{eq:h-clust-unif-close}
\mu_N \left(\sigma^1,\sigma^2\in C_{\alpha,N} :(\sigma^1,\sigma^2)\leq  q_{\abs{\alpha}}-\eps\right)\leq \delta.
\end{equation}
\item Points in cousins are uniformly far: for every $\alpha\nsim\beta$ in $\tau_{m_N}$, if $\gamma\prec\alpha$ is such that
$\abs{\gamma}=\abs{\alpha\wedge\beta}+1$ and similarly for $\eta$ and $\beta$, then %%
\begin{equation}\label{eq:h-clust-unif-far}
\mu_N\left(\sigma^1\in C_\gamma,\sigma^2\in C_\eta :(\sigma^1,\sigma^2) \geq q_{\abs{\gamma\wedge\eta}+1} +\eps\right) \leq \delta.
\end{equation}
\end{enumerate}
\end{defn}

Using this language, we restate the definition of approximately ultrametricity as a sequence
 of measures that admit sets that are increasingly exhausting and clustering with high probability. 
 
 \begin{defn*}
 A sequence $\{\mu_N\}_{N=1}^\infty$ of random measures on the unit ball of $\ell_2$ 
 such that is said to be 
 \emph{approximately ultrametric with respect to $\zeta$} if for every $\zeta$-admissible sequence, 
$\{q_k\}_{k=1}^r$,  there is a sequence of finite rooted trees of depth $r$, $\{\tau_N\}$, 
and sequences $a_N$, $b_N$, and $\eps_N$ all tending to $0$ such that with probability tending to one, 
there are sets $\{C_{\alpha,N}\}_{\alpha\in\tau_N}$ such that they are an $(\eps_N,0)$-hierarchically exhaustion of $\mu_N$
and are $(a_N,b_N)$-hierarchically clustering for $\mu_N$ with respect to the sequence $\{(q_k)\}$.
 \end{defn*}

The $\{C_{\alpha,N}\}$ in the above definition 
are to be directly compared with the
pure states of physicists \cite{MPV87}. 
Recall that these ``pure states'' arrange hierarchically into the equivalence classes of replica (i.i.d. draws)
$$\sigma^1\sim\sigma^2 \iff (\sigma^1,\sigma^2)\geq q_k.$$ 
In the language above, this is  the partitioning property of ultrametric spaces. 
We cannot expect such overlap based equivalence classes to form at 
finite $N$ as can be
seen by constructing a sequence of overlap distributions that are almost RPC's at large but 
finite N but fail to satisfy this clustering. This is due to the issue mentioned before, namely there can 
be points in $C_\alpha$ that are so close to those in $C_\beta$ that their balls cut in to both sets. 
Instead we get that on average such a clustering happens and that such exceptional points become 
increasingly rare in the limit.

\subsection{Dovbysh-Sudakov measures and Consequences of the Ghirlanda-Guerra Property}\label{sec:GGI}
The key element of the following analysis is that the sequence of measures satisfies
the Approximate Ghirlanda-Guerra identities. We briefly summarize the structure theory of such sequences. 
For a more in-depth survey see \cite{PanchSKBook,Panch12}. 

We begin with the following definitions. Fix $\mu$ a random probability measure on the unit ball 
of $\ell_2$. Draw $(\sigma^i)$ iid from $\mu$ and form the doubly infinite
array of pairwise inner products 
\[
R=\left(R_{ij}\right)_{i,j,\geq 1}.
\]
We call the pair $(R,\mu)$ a \emph{ROSt} \cite{Arg08}. We call the  array $R$ the \emph{Gram-DeFinetti}
array and $\mu$ the \emph{Dovbysh-Sudakov measure} of the ROSt. 
The array $R$ is \emph{weak exchangeable}, that is,
if $\pi$ is a permutation of $\N$, then
\[
(R_{ij})\eqdist(R_{\pi(i)\pi(j)}).
\]
In general, we call a random doubly infinite array whose minors are positive 
semi-definite a \emph{Gram-DeFinetti} array. An important property of ROSts 
is contained in the Dovbysh-Sudakov theorem which we state in a simplified form.

\begin{prop*}[Dovbysh-Sudakov]
For any Gram-DeFinetti array $R$ such that $\abs{R_{ij}}\leq 1$, there
is a ROSt $(\tilde{R},\mu)$ and a random probability measure $\nu$ on $\R_+$ such
that if $a_i$ are iid drawn from $\nu$, then 
\[
(R_{ij})\eqdist (\tilde{R}_{ij} + a_{i}\delta_{ij}).
\]
\end{prop*}

Let $\mu_N$ be a sequence of random probability measures on the unit ball of $\ell_2$
that satisfy the AGGI's. Denote the laws of the Gram-DeFinetti arrays by 
$P_N$. By compactness, there is a $P$ such that 
$P_N\rightarrow P$
weakly and such that $P$ is the law of a Gram-DeFinetti array. Let $(R,\mu)$
be the ROSt corresponding to $P$ given by the Dovbysh-Sudakov theorem. In the case that
$\{\mu_N\}$ arises from a sequence of Gibbs measures as in \prettyref{sec:applications}, we call this
$\mu$ the \emph{limiting Dovbysh-Sudakov measure} of the sequence $\{\mu_N\}$. This is precisely
the Asymptotic Gibbs Measure of Panchenko.
The Dovbysh-Sudakov Measure $\mu$ must satisfy the \emph{Ghirlanda-Guerra Identities}:
for all $n$, bounded Borel $f$, and continuous $\psi$
\[
\E\gibbs{f(R^n)\psi(R_{1,n+1})} = \frac{1}{n}\left(\E\gibbs{f(R^n)}\E\gibbs{\psi(R_{12})}
+\sum_{k=2}^n \E\gibbs{f(R^n)\psi(R_{1,k})}\right)
\]
where $R^n$ is the $n$-th minor of $R$. Measures that satisfy the Ghirlanda-Guerra identities have the following
properties. 
\begin{prop}\label{prop:GGI-cons}\cite{PanchSKBook}
Let $\mu$ satisfy the Ghirlanda-Guerra identities. Then:
\begin{itemize}
\item The measure is concentrated on a sphere: if $q_*$ is the supremum of the support of $\zeta$ the overlap
distribution for $\mu$, then $\mu(\norm{\sigma}=q_*)=1$ almost surely.
\item Talagrand's Positivity Principle: $\mu^{\tensor 2}(R_{12}\in [-1,0))=0$ almost surely.
\item Panchenko's Ultrametricity Theorem: the support of $\mu$ is almost surely ultrametric. That is,
\[
\E \mu^{\tensor3}(R_{12}\leq R_{13}\wedge R_{23})=0
\]
\item Baffiano-Rosati theorem: the law of $\mu$ is uniquely specified by its overlap distribution 
(modulo partial isometries of separable Hilbert space).
\end{itemize}
\end{prop}

We end this section with the following well-known consequence of the Ghirlanda-Guerra
 Identities. (See \cite[Section 2.4]{PanchSKBook}, particularly the discussion regarding 
the $\kappa$ approximation to $R$, and \cite{PanchHEpure2013}.) By a $q$-ball, we mean a set of
the form $\{\sigma\in B(0,1):(\sigma,\sigma_0)\geq q\}$ for some $\sigma_0$
 
\begin{fact}\label{fact:RPC}
Let $\mu$ satisfy the Ghirlanda-Guerra Identities. Let $\zeta(\cdot)=\E\mu^{\tensor 2}((\sigma^1,\sigma^2)\in\cdot )$. Let $\{q_k\}_{k=1}^r$ be $\zeta$-admissible. Partition
the support of $\mu$ as follows. Let $B_n$ be a sequence of $q_1$-balls 
that partition the support of $\mu$. Let $B_{\alpha n}$ be $q_{\abs{\alpha}}$-balls 
such that $B_\alpha =\cup_n B_{\alpha n}$. Finally , let $V_\alpha$ be 
the $\mu$-masses of these balls arranged in standard order. 
The law of these weights is distributed like those of an
RPC with the overlap distribution with parameters $\zeta_k-\zeta_{k-1}=\zeta[q_k,q_{k+1})$.
In particular, there are infinitely
many of them at each level and they have almost surely non-zero weights. 
\end{fact}
\prettyref{fact:RPC} is well-known in the literature (see 
for example \cite{PanchHEpure2013}), however to our knowledge a proof has never been
published, so for the convenience of the reader we prove this in the Appendix.

\subsection{Outline of Proofs of Main Results} \label{sec:outline}

In this section we outline the strategy of the proofs of
\prettyref{thm:main-result-1} and
\prettyref{thm:main-result-2}.  Before we start, we fix some notation.  
Fix a sequence $\{\mu_N\}$ that satisfies the conditions of
\prettyref{thm:main-result-1}. Let $\{P_N\}$ denote the corresponding 
sequence of laws of the Gram-DeFinetti arrays corresponding to $\mu_N$. 
Since
$\zeta_N \rightarrow \zeta$
weakly by assumption (see \prettyref{thm:main-result-1} for this notation), 
and $\mu_N$ satisfies the AGGIs,  it follows from the Baffiano-Rosati theorem
(see \prettyref{prop:GGI-cons})  that there is a unique $P$ corresponding to
a Gram-DeFinetti array such that 
$P_N \rightarrow P$
weakly. Let $\mu$ be the Dovbysh-Sudakov measure corresponding to $P$. 
After possibly enlarging the background
probability space, we couple 
this sequence $\{\mu_N\}$ with $\mu$ and we take them all to be living on
a single space $(\Omega,\cF,\prob)$ which we call 
\emph{the background space of disorder}. 

We begin first with the proof of \prettyref{thm:main-result-1}. We begin the proof in \prettyref{sec:ballweights}  by showing that if we sample 
random balls in the support of $\mu_N$, then the collection of their masses
converges in law to those corresponding to balls drawn from $\mu$. We make
this precise as follows

Draw $\bsig=(\sigma^{\alpha})_{\alpha\in\cA_{r}}$ i.i.d from $\mu_{N}$ 
and consider the collection of sets $\{B_{\alpha}\}_{\alpha\in\cA_{r}}$
with 
\begin{equation}\label{eq:B-Def}
B_{\alpha}(\bsig)=\bigcap_{\beta\prec\alpha}B(\sigma^{\beta},q_{\abs{\beta}}),
\end{equation}
where 
\[
B(\sigma,q)=\{x\in B_{\ell_2}(0,1) :(\sigma,x)\geq q\}.
\]
(Since we only consider its intersection
with a sphere, we call it a $q$-ball.)
That is, it is the intersection of the $q_{\abs{\alpha}}$-ball of
$\sigma^{\alpha}$ with the corresponding balls of its ancestors. 

Let $F^r_{2}$ denote the set of subsets of vertices in $\cA_{r}$ of size at most 2. 
Consider the following set of weights
\begin{equation}\label{eq:W_N-Def}
(W_{E,N})_{E\in F^r_{2}}=\left\{ \mu_{N}\left(\bigcap_{\alpha\in E}B_{\alpha}\right)\right\} _{E\in F^r_{2}}
\end{equation}
which is a sequence $\mathbf{W}_{N}$ of random variables taking values in $[0,1]^{F^r_2}$.
We denote their laws by $Q_N$. We define $\bfW$ similarly for $\mu$ the Dovbysh-Sudakov measure, and 
denote the corresponding law $Q$. This sampling structure is intended to mimic the hierarchical structure of the RPC
while also storing additional data about intersections of relevant balls.
Then we have: 
\begin{lem}\label{lem:W-conv}
Fix a $\zeta$-admissible sequence $\{q_k\}_{k=1}^r$. The sequence $\mathbf{W}_{N}$ defined in \prettyref{eq:W_N-Def} converges to
$\mathbf{W}=(W_{E})_{E\in F^r_{2}}$ in distribution.
\end{lem}

The next  step in the proof, in \prettyref{sec:exh-sets}, is to show that there are 
$B_\alpha$ as above in the support of $\mu_N$ that 
define an $(\eps,\delta)$-hierarchical exhaustion.

\begin{prop}\label{prop:exhaustion-1}
For every $r$ and $\zeta$-admissible sequence $\{q_k\}_{k=1}^r$, we have that for every $\epsilon,\delta$ positive, 
\begin{equation}
\begin{aligned}
\lim_{N\rightarrow\infty} \prob(\exists\tau_{N,r}\in\cT_r:&\exists \text{a }(\eps,\frac{\delta}{\abs{\tau_{N,r}}^2})-\text{hierarchical exhaustion of } \mu_N ) =\\
&\prob(\exists\tau_{N,r}\in\cT_r:\exists \text{a }(\eps,\frac{\delta}{\abs{\tau_{N,r}}^2})-\text{hierarchical exhaustion of } \mu )=1
\end{aligned}
\end{equation}
\end{prop}
We then ``clean" up the above sets by removing intersections so that we can take $\delta=0$ in the above. In particular, we have:
\begin{cor}\label{cor:exhaustion-2}
For every $\eps$ positive, $r$ and $\zeta$-admissible sequence $\{q_k\}_{k=1}^r$,
\[
\lim_{N\rightarrow\infty} \prob(\exists\tau_{N,r}:\exists(\epsilon,0)-\text{exhaustion of }\mu_{N})=1.
\]
\end{cor}
In \prettyref{subsub:Regularity}, we regularize 
 $\tau_{N,r}$, by showing that for $N$ large enough,  we can take
$\tau_{N,r}$ is of the form $\tau_{m,r}$ for some $m$ (recall this notation from
\prettyref{sec:trees}).

In \prettyref{sec:h-clust}, we show that we can take the above sets to be hierarchically clustering with high probability. 
This will follow from Panchenko's Ultrametricity
theorem (\prettyref{prop:GGI-cons}).
We combine these results  in \prettyref{thm:main-result-1-proof} to show that the sequence $\{\mu_N\}$ admits an increasingly more exhausting and clustering sequence of clusters. As this applies to any admissible sequence, this concludes the proof of \prettyref{thm:main-result-1}

The proof of \prettyref{thm:main-result-2} follows in two steps. 
First, we show that the collection of masses corresponding to the leaves, rearranged by mass converges to a Poisson-Dirichlet process (see \prettyref{thm:clust-reg-pd}). 
This follows from a standard application of Talagrand's 
Identities for the Poisson-Dirichlet process,
after making an observation about approximation the indicator functions of
sets in a clustering by functions of the overlap. (See \cite{TalBK03,PanchSKBook, ArgZind14} 
for explanations and examples of the former technique.)

We then introduce the space of cascade of depth $r$ of which the collection of masses $(Y^N_\alpha)$ are an element for
each $N$. We present an encoding of a cascade into a ROSt and show that for
the distributional limits of a sequence of cascades, the corresponding sequence of ROSts must also converge. 
We then characterize the limit of the sequence of ROSts corresponding to the $(Y^N_\alpha)$ and find that the limit object
is unique and is given by an RPC. We then use this to conclude that the sequence of 
masses converge in distribution to a unique limit that is given by the weights of an RPC.
(This argument is to be compared with the convergence of structural distributions for mass partitions in \cite{Bert06}.)

\section{Weak Convergence of the ball weights}\label{sec:ballweights}
For the rest of this \documentname, we work with a fixed $r$ and suppress the notation in $r$.
In this section, we prove the weak convergence of the weights described in \prettyref{eq:W_N-Def}. 
using the method of moments.
The main observation is that computing the moments of $W_E$ is equivalent to computing probabilities
of certain events regarding replica overlaps. 

We begin with the following observation.
Since $F^r_{2}$ is countable,  the product space $\left([0,1]^{F_{2}},\tau_{prod}\right)$
is compact Hausdorff. Consider the measurable space $\left([0,1]^{F_{2}},\cB\right)$
where $\cB$ is the Borel $\sigma$-algebra for the product topology. A standard argument then gives that moments are convergence determining class. 
The proof is an application of Stone-Weierstrass so it is omitted.
% placed in the appendix.
\begin{lem}
\label{lem:conv-dist-F2}
Then a sequence of measures $\nu_{N}\in\cM([0,1]^{F_{2}},\cB)$ converges
weakly to $\nu$ iff for all finite subsets $E\subset F_{2}$, and
(finite) sequences $\{n_{B}\}_{B\in E}\subset\N$ , 
\[
\int\prod_{B\in E}\pi_{B}(x)^{n_{B}}d\nu_{N}\rightarrow\int\prod_{B\in E}\pi_{B}(x)^{n_{B}}d\nu.
\]
 \end{lem}

Now consider the class of sets $\mathscr{A}$ defined as follows. Choose $M,n$ finite, a collection
of sets $\{E_k\}_{k=1}^M\subset F_2$, and a partition $\cup_{k=1}^{M}I_k = [n]$ such that $\abs{I_k}=n_k$.
Let
\[J=\left\{ (\beta,j,l):j\in I_{l},\beta\precsim\alpha,\alpha\in E_{l},l\leq M\right\}\] 
and consider the sets
\begin{equation}\label{eq:F2-cty-set}
A=\bigcap_{(\beta,j,l)\in J}\left\{ R(\sigma^{\beta},\tilde{\sigma}^{j})\geq q_{\abs{\beta}}\right\}. 
\end{equation}
with $\zeta$-admissible $\{q_k\}$. Let
$\mathscr{A} = \{A \text{ of the form \prettyref{eq:F2-cty-set}}\}. $ We then have the following lemma
 whose proof, we omit as it follows immediately from the definition of admissible sequences.

\begin{lem}\label{lem:F2-cty-set}
The elements of $\mathscr{A}$ are $P$-continuity sets.
\end{lem}
The  proof  of \prettyref{lem:W-conv} can now be completed. We restate it for the convenience of the reader. 
\begin{lem*}[\ref{lem:W-conv}]
For every $r$ and $\zeta$-admissible $\{q_k\}_{k=1}^r$, the sequence $W_{N}$ converges to $W$ in law.
\end{lem*}
\begin{proof}
By \prettyref{lem:conv-dist-F2}, it suffices to show convergence
of joint moments 
\[
\E\prod_{k=1}^{M}W_{E_{k},N}^{n_{k}}\rightarrow\E\prod W_{E_{k}}^{n_{k}}.
\]
If we let $\cup I_{k}$ 
and $J$ be as above and let $(\tilde{\sigma}_N^i)$ be drawn i.i.d. from 
$\mu_{N}$, 
\begin{align*}
\E\prod_{k=1}^{M}W_{E_{k},N}^{n_{k}}=\E\left\langle \prod_{k\leq M}\mu_{N}\left(\bigcap_{\alpha\in E_{k}}B_{\alpha}\right)^{n_{k}}\right\rangle 
 =\E\left\langle \prod_{(\beta,j,l)\in J}\indicator{R(\sigma^{\beta},\tilde{\sigma}^{j})\geq q_{\abs{\beta}}}\right\rangle
 =P_{N}(A)
\end{align*}
for some $A\in\mathscr{A}$. By \prettyref{lem:F2-cty-set}, it then follows that
\[
P_{N}(A)  \rightarrow P(A) = \E\prod W_{E_{k}}^{n_{k}}
\]
as desired.
\end{proof}

\section{Exhausting sets}\label{sec:exh-sets}
In this section, we prove the (high probability) existence of hierarchical 
exhaustions for the sequence of $\mu_N$. We begin the section with some measure theoretic preliminaries,
and prove the main technical lemmas. We then prove the existence of the exhaustions, and we conclude with some results on the regularity of these exhaustions.

\subsection{Some measure theoretic preliminaries}\label{sec:mt-prelims}

Recall the definition of $\bfW(\omega)$ from \prettyref{eq:W_N-Def}. If we view 
the pair $(\bsig,\mu)$ through the map
$\omega\mapsto (\bsig(\omega),\mu(\omega))$,
and similarly for $(\bsig_N,\mu_N)$, and we consider the map 
\[
\bftW(\sigma,\mu) = (\mu(\bigcap_{\alpha\in E} B_\alpha(\bsig))_{E \in F_2},
\]
then $\bfW(\omega)=\bftW(\bsig(\omega),\mu(\omega))$.
That the relevant maps are measurable can be seen by a monotone class argument.

We now define events that are related to the existence of hierarchical exhaustions.
On the space $[0,1]^{F_2}$, we define for every finitely shaped  tree $\tau$, and $\eps$ and $\delta$ positive the open sets 
\begin{equation}\label{eq:Ated}
\begin{aligned}
A_{\tau,\epsilon,\delta} 
&=\bigcap_{k=1}^{r}\left\{\sum_{\substack{\abs{\alpha}=k\\\alpha\in\tau}}x_{\{\alpha\}}>1-\epsilon\right\}
\bigcap_{\alpha\in\tau\setminus\partial\tau}\left\{  x_{\{\alpha\}}-\sum_{\beta\in child(\alpha)}x_{\{\beta\}}\in(0,\epsilon)\right\}\bigcap\left\{ 0\leq \sum_{\substack{\alpha\nsim\beta\\\alpha\beta\in\tau}}x_{\alpha,\beta} <\delta/\abs{\tau}^{2}\right\}. 
\end{aligned}
\end{equation}
That these sets are open subsets of $[0,1]^{F_2}$ as can be seen from the fact that $\abs{\tau}<\infty$ (note that in the last
set you should consider the relevant map as going into $\R_+$ so that $[0,\delta/\abs{\tau}^2)$ is open.

Pulling back through the above maps, we get 
\begin{equation}\label{eq:Enu}
E_{\tau,\eps,\delta}(\nu)=\left ( \bftW(\cdot,\nu)\right)^{-1}\left(A_{\tau,\eps,\delta}\right)
\end{equation}
 which is a subset of $B_{\ell_2}(0,1)^{\cA_r}$ for any  measure $\nu$ and 
\begin{equation}\label{eq:E}
\tilde{E}_{\tau,\eps,\delta}=\bfW^{-1}\left(A_{\tau,\eps,\delta}\right)
\end{equation}
which is a subset of the background space of disorder $\Omega$. We define $\tilde{E}_{\tau,\eps,\delta}^N$
similarly with $\bfW_N$. By the same token, we define $A_{\tau^i,\eps,\delta}$ and its pull-backs using $\tau^i$
(recall this notation from \prettyref{sec:trees}).
Let
\begin{equation}
A_{\eps,\delta}=\bigcup_{\tau\in\cT_{r}}\bigcup_{i=1}^{\infty}A_{\tau,\epsilon,\delta}^{i},
\end{equation}
which is also an open subset of $[0,1]^{F_2}$ and define $E_{\eps,\delta}\left(\mu_N\right)$ and  $\tilde{E}^N_{\eps,\delta}$ analogously.

Notice that by definition, for any $\eps,\delta$ positive, finitely shaped $\tau$ and random probability measure $\nu$, 
the set
\[
E_{\tau,\eps,\delta}(\nu)
\subset\left\{\bsig\in B_{\ell_2}(0,1): (B_\alpha)_{\alpha\in \tau} \text{ is an } (\eps,\frac{\delta}{\abs{\tau}^2})-
\text{hierarchical exhaustion of }\nu \right\}
\]
so that if we union over $\tau's$,
\[
E_{\eps,\delta}(\nu)\subset\left\{\bsig:\exists \tau\in\cT_r, i\in\N:(B_\alpha)_{\alpha\in\tau^i}
\text{ is an }(\eps,\frac{\delta}{\abs{\tau}^2})-\text{hierarchical exhaustion of } \nu\right\}.
\]
Thus if we can show,
$\prob(\tilde{E}_{\eps,\delta})=1$,
then we know that almost surely there is an $(\eps,\delta/\abs{\tau}^2)$-hierarchical exhaustion of $\mu$ for some finitely shaped $\tau$.

We conclude this section with the following observation which follows
immediately from measure disintegration \cite{Kal97}.
\begin{lem}\label{lem:disintegrate}
We have that
\[
Q(A_{\eps,\delta}) = \prob(\tilde{E}_{\eps,\delta})=\E \mu^{\tensor\infty}\left(E_{\eps,\delta}(\mu)\right).
\]
\end{lem}

\subsection{Technical lemmas}
We remind the reader that $\mu$ is the Dovbysh-Sudakov measure for $P$ the limit of the sequence $P_N$. In the following 
sections we fix a $\zeta$-admissible sequence $\{q_k\}$. We prove quantitative versions of many of these lemmas in \prettyref{sec:quant-version}.
\begin{lem}\label{lem:tau-exists}
Let $\mu$ be as in \prettyref{sec:outline}. Then $\prob$-a.s.there is a shaped tree $\tau$ such that , 
\[
\mu^{\tensor\cA_r}\left(E_{\tau,\epsilon,\delta}(\mu)\right)>0.
\]
\end{lem}

\begin{proof}
Draw $\mu$. By \prettyref{fact:RPC}, we know that corresponding to the admissible sequence $\{q_k\}$,
there is a nested sequence of balls and weights $(U_\alpha,v_\alpha)$ with
\[
\mu (U_\alpha) =v_\alpha >0
\]
a.s. where $U_\alpha$ has radius $q_{\abs{\alpha}}$.

Observe that 
$E_{\tau,\epsilon,\delta}(\mu) \supset E_{\tau,\epsilon,0}(\mu)$,
where by $\delta =0$ we mean that the last summand in \prettyref{eq:Ated} is $0$. 
It then suffices to show that for some $\tau=\tau(\mu)$ the latter set has
positive $\mu$-mass. We choose $\tau\subset\cA_{r}$
as follows. 

Pick the smallest $m_{1}$ so that 
\[
\sum_{k=1}^{m_{1}}v_{k}>1-\epsilon.
\]
 For each $\abs{\alpha}=1$ such that $\alpha \leq m_1$, let $m_\alpha$ be the smallest $m$ such that 
\[
\sum_{k=1}^{m_\alpha} v_{\alpha k}\geq v_\alpha -\eps
\]
and let $\tilde{m}_{2}=\max{m_\alpha}$.
Let
\[
m_2=\min\{m\geq \tilde{m}_2:\sum_{k\leq m_1}\sum_{n=1}^m v_{kn}>1-\eps\}.
\]
Construct $m_k$ similarly for $k\in \{3,\ldots,r\}$. Let $\tau$ be the $(m_1,\ldots,m_r)$-shaped tree. 

Notice that
\begin{equation}\label{eq:e-tau-pos}
\mu^{\tensor\cA_r}(E_{\tau,\epsilon,\delta}(\mu))\geq\mu^{\tensor\cA_r}(E_{\tau,\epsilon,0}(\mu))\geq\prod_{\alpha\in\tau}v_{\alpha}>0.
\end{equation}
where the second inequality comes from noting that this is the chance
that $\sigma^{\alpha}$ lands in the $U_{\alpha}$. 
\end{proof}
\begin{lem}\label{lem:exhausting-tau}
We have %%MARK 
\[
\mu^{\tensor\cA_{r}}\left(E_{\eps,\delta}(\mu)\right)=1
\]
$\prob$-a.s.
\end{lem}
\begin{proof}
Draw $\mu$ and define $\tau$ as in \prettyref{lem:tau-exists}. Form the $\tau$-pruning of $\cA_r$ which
we denote by $(\sigma^{\tau^i})_{i=1}^\infty$ where we abuse notation and let $\sigma^\tau = (\sigma_\alpha)_{\alpha\in\tau}$.

Denote by $\pi E_{\tau,\eps,\delta}$ the projection of $E_{\tau,\eps,\delta}(\mu)$ on to the coordinates 
indexed by  $\tau$. Note that $E_{\tau,\eps,\delta}(\mu)=\pi E_{\tau,\eps,\delta}\times B(0,1)^{\cA_r\setminus\tau}$,
interpreting the product suitably.
Let $\cE_n$ be the empirical measure % possibly delete
\[
\cE_n = \frac{1}{n}\sum_{i=1}^n \delta_{\sigma^{\tau^i}}
\]
which is a measure on $B(0,1)^\tau$. By the Law of Large Numbers applied conditionally 
on $\mu$,
\[
\lim \cE_n(\pi E_{\tau,\eps,\delta}(\mu)) = \mu^{\tensor\tau}(\pi E_{\tau,\eps,\delta})>0
\]
a.s.. As a consequence $\mu$-a.s. there is an $I$ such that $\sigma^{\tau^I}\in \pi E_{\tau,\eps,\delta}(\mu)$, so
that $\bsig\in \cup E_{\tau^i,\eps,\delta}(\mu)$ $\mu$-a.s.. Thus the first equality holds. 
The second equality holds $\prob$-a.s. by set containment.\end{proof}

\subsection{Hierarchical Exhaustion: Proofs}
In this section we prove \prettyref{prop:exhaustion-1} and \prettyref{cor:exhaustion-2}. 
For the convenience of the reader, we restate these results in the above language.
\begin{prop*}[\ref{prop:exhaustion-1}]
We have that for every $\epsilon,\delta$ positive, 
\[
\lim Q_{N}(A_{\epsilon,\delta})=Q(A_{\epsilon,\delta})=1.
\]
\end{prop*}
\begin{proof}
Since $A_{\epsilon,\delta}$ is open, it follows  
by weak convergence of $Q_{N}$, the law of the collection 
$(W_{E,N})_{E\in F_2}$,
that 
\[
1\geq\liminf Q_{N}(A_{\eps,\delta})\geq Q(A_{\eps,\delta}).
\]
By \prettyref{lem:disintegrate} and \prettyref{lem:exhausting-tau},  
$Q(A_{\eps,\delta})=1.$
 \end{proof}
\begin{cor*}[\ref{cor:exhaustion-2}]
For every $\epsilon$ positive, we have that 
\[
\lim_{N\rightarrow\infty} \prob(\exists(\epsilon,0)-\text{exhaustion of }\mu_{N})=1.
\]
\end{cor*}
\begin{proof}
By \prettyref{prop:exhaustion-1}, we know that for $\delta=\epsilon/2$ and
 for every choice of $\eta>0$, 
\[
Q_{N}(A_{\epsilon/2,\epsilon/2})\geq1-\eta
\]
for $N$ large enough. This means that there is a $\tau$ and a collection of sets 
$\{B_{\alpha}\}_{\alpha\in\tau}$ as per \prettyref{eq:B-Def} with 
\[
\sum_{\abs{\alpha}=k}\mu_{N}(B_{\alpha})\geq1-\epsilon/2
\text{ and }
\sum_{\alpha\nsim\beta}\mu_{N}\left(B_{\alpha}\cap B_{\beta}\right)\leq\frac{\eps}{2\abs{\tau}^{2}}
\]
such that 
\[
\frac{\epsilon}{2}\geq\mu_{N}(B_{\alpha})-\sum_{\beta\in child(\alpha)}\mu_{N}(B_{\beta}).
\]
Using $B_{\alpha}$ we construct $C_\alpha$ given by
\[
C_{\alpha}=B_{\alpha}\backslash\left(\bigcup_{\beta\nsim\alpha}B_{\beta}\right).
\]
The $C_{\alpha}$'s have the desired inclusion structure and their intersections are null. 
The measure properties then follow immediately from inclusion-exclusion arguments.
\end{proof}

\subsection{Regularity of Exhausting Sets\label{subsub:Regularity}}
To start out, we need to show that we can regularize $\tau$  at the level of the RPC.
\begin{lem}\label{lem:q-a-lb}
For every $\eta,\eps,\delta$, there is an $m(\eta,\eps,\delta)$ such that 
\[
Q(\cup_{i=1}^{\infty}A_{\tau_{m},\epsilon,\delta}^{i})\geq1-\eta.
\]
\end{lem}
\begin{proof}
By disintegration  %(argue as in \prettyref{lem:disintegrate}),
\begin{align*}
Q(\cup_{i}A_{\tau_{m},\epsilon,\delta}^{i})=
\E\mu^{\tensor\infty}(\cup_{i}E_{\tau_{m},\epsilon,\delta}^{i}(\mu))
\geq\E\mu^{\tensor\infty}(\cup_{i}E_{\tau_{m},\epsilon,\delta}^{i}(\mu))\indicator{\mu^{\tensor\infty}(\cup_{i}E_{\tau_{m},\epsilon,\delta}^{i}(\mu))=1} \\
= P(\mu^{\tensor\infty}(\cup_{i}E_{\tau_{m},\epsilon,\delta}^{i}(\mu))=1)
\end{align*}
so it suffices to show that
\[
\prob(\mu^{\tensor\cA_r}(\cup_{i}E_{\tau_{m},\epsilon,\delta}^{i}(\mu))=1)\geq1-\eta.
\]
By the Law of Large Numbers, it suffices to prove
\[
\prob(\mu^{\tensor\tau_m}(E_{\tau_{m},\epsilon,\delta}(\mu))>0)>1-\eta.
\] 

Let $(U_\alpha,v_\alpha)$ be the decomposition of $\mu$ into nested balls as in \prettyref{lem:tau-exists}.
Then we know that 
\[
\sum_{\abs{\alpha}=r}v_{\alpha}=1
\]
a.s.. Let $\tau_{m}$ be the $(m,\ldots,m)$-shaped tree, then
\[
\sum_{\alpha\in\partial\tau_{m}}v_{\alpha}\rightarrow1
\]
as $m\rightarrow\infty$ a.s..
Thus for every $\epsilon,\eta>0$, there is an $m(\epsilon,\eta)$ such that
\[
\prob(\sum_{\alpha\in\partial\tau_{m}}v_{\alpha}>1-\epsilon)\geq1-\eta
\]
by standard arguments. Then 
\[
\prob(\mu^{\tensor\tau_{m}}(E_{\tau_{m},\epsilon,\delta}(\mu))>0)\geq1-\eta
\]
by the same argument as for  \prettyref{eq:e-tau-pos} in \prettyref{lem:tau-exists} by set incluson and the fact that
$v_\alpha = \sum v_{\alpha n}$. 
\end{proof}

\begin{prop}\label{prop:regularity}
For every $\eta,\eps,\delta$, there is an $M=M(\eps,\delta,\eta)$, an $m=m(\eta,\eps,\delta)$, and an $N_0(\eta,\eps,\delta)$ such that for $N\geq N_{0}$ 
\[
Q_{N}(\cup_{i=1}^{M}A_{\tau_{m},\epsilon,\delta}^{i})\geq1-\eta
\]
\end{prop}
\begin{proof}
Fix $m$ from \prettyref{lem:q-a-lb}. By monotonicity,  
\[
\lim_{M} Q(\cup_{i=1}^{M}A_{\tau_{m},\epsilon,\delta}^{i})=Q(\cup_{i=1}^{\infty}A_{\tau_{m},\epsilon,\delta}^{i})\geq1-\eta/4
\]
so that there is an $M(\eta,\eps,\delta)$  such that for $\tilde{M}\geq M(\eps,\eta,\delta)$,
\[
Q(\cup_{i=1}^{\tilde{M}}A_{\tau_{m},\epsilon,\delta}^{i})\geq1-\eta/2
\]
as desired. 

Since $\cup_{i=1}^{M}A_{\tau_{m},\epsilon,\delta}$ is open 
it follows by weak convergence of $Q_N$ that
\[
\liminf Q_{N}(\cup_{i=1}^{M}A_{\tau_{m},\epsilon,\delta})\geq Q(\cup_{i=1}^{M}A_{\tau_{m},\epsilon,\delta})\geq1-\eta/2
\]
 so that for $N$ large enough 
\[
Q_{N}(\cup_{i=1}^{M}A_{\tau_{m},\epsilon,\delta})\geq 1-\eta
\]
\end{proof}

\section{Hierarchical Clustering}\label{sec:h-clust}
In this section, we prove that the clusters constructed above are hierarchically clustering for $\mu_N$. We then conclude by
proving \prettyref{thm:main-result-1}

In order to show these results, we introduce the following quantities. Let
\begin{equation}\label{eq:f-def}
f^N_{k,\epsilon}(\sigma)=\mu_N^{\tensor2}\left(\sigma^1,\sigma^2:(\sigma^{1},\sigma^{2})\leq q_{k}-\epsilon,(\sigma^{1},\sigma)\geq q_{k},(\sigma^{2},\sigma)\geq q_{k}\right)
\end{equation}
which will encode \prettyref{eq:h-clust-unif-close}. 
Similarly, for each $\alpha\nsim\beta\in\cA_r$, define
\begin{equation}\label{eq:g-def}
g^N_{\alpha,\beta,\epsilon}(\bsig)=\mu_N^{\tensor2}\left(\sigma^1,\sigma^2:R_{12}\geq q_{k}+\epsilon,R_{1\gamma}\geq q_{k},R_{2\delta}\geq q_{k},R_{2\gamma}< q_{k},R_{1\delta}< q_{k}\right).
\end{equation}
This encodes \prettyref{eq:h-clust-unif-far}.

\subsection{Preliminary lemmas}
Here we record some useful consequences of Panchenko's Ultrametricity Theorem.
\begin{lem}
\label{lem:ult} 
Let $\epsilon_{n}=\frac{1}{2^{n}}$. There is a sequence $N_1(n)$
such that $N\geq N_1(n)$ gives 
\[
\E\mu_{N}^{\tensor 3}(R_{12}\leq R_{13}\wedge R_{23}-\epsilon_{n})\leq\epsilon_{n}.
\]
and for all $k\in[r]$,
\[
\E\mu_N^{\tensor 3}(R_{12}\geq q_k + \eps_n, R_{13}\geq q_k, R_{23}<q_k)
\leq \eps_n
\]
\end{lem}
%%%MARK
\begin{proof}
Let $\bar{P}_N$ and $\bar{P}$ be $P_N$ and $P$ restricted to the coordinates $R_{12}, R_{13}$, and $R_{23}$.

The first claim follows from the ultrametricity theorem after noting that
the relevant sets are $\bar{P}$ continuity sets for each $n$. 
To see the second claim, let 
\[
A^k_\eps =\{ x_1 \geq q_k + \eps, x_2 \geq q_k, x_3 <q_k\},
\]
which is a subset of $[-1,1]^3$.  Note that 
\begin{equation}\label{eq:ult-lem-interp-eq}
\bar{P}_N(A_\eps^k)=\E \mu^{\tensor 3}_N(R_{12}\geq q_k+\eps, R_{13} \geq q_k, R_{23} <q_k)
\end{equation}
and similarly for $\bar{P}(A_\eps^k)$.  Since $A^k_{\eps}\subseteq A^k_0$, the result follows by
set inclusion after noting that the latter set is a continuity set whose measure vanishes in the limit. 
As we are only considering finitely many sets for fixed $n$, we can chose a single $N_1(n)$ such that all
of the above inequalities happen simultaneously.
\end{proof}
%%%ENDMARK

\begin{lem}\label{lem:fk-N}
Let  $f^N_{k,\epsilon}(\sigma^{\alpha})$ be as in \prettyref{eq:f-def}. 
Let $m$ and $M$ be fixed and $\epsilon_{n}=\frac{1}{2^{n}}$.
Then for $N\geq N_1(n)$, we have 
\[
\E \mu^{\tensor\infty}_{N}\left(\sum_{\alpha\in\tau_{m}^{i}}f_{k\sqrt{\epsilon_{n}}}(\sigma^{\alpha}_N)<\sqrt{\epsilon_{n}},\forall i\in[M]\right)\geq1-rMm^{r}\sqrt{\epsilon_{n}}.
\]
\end{lem}
\begin{proof}
We suppress the superscript $N$ and subscript $\eps$. Observe that
\[
\gibbs{ f_{k}(\sigma_N)} =\mu_N^{\tensor3}(R_{12}\leq q_{k}-\epsilon,R_{13}\geq q_k,R_{23}\geq q_{k})
 \leq\mu_N^{\tensor3}\left(R_{12}\leq R_{13}\wedge R_{23}-\epsilon\right)
\]
for each $k$. Now if we let $\epsilon=\epsilon_{n}$ be as in \prettyref{lem:ult},%%%MARK
then if $N\geq N_1(n)$, we have that by linearity,
\[
\E\left\langle \sum_{i=1}^{M}\sum_{\alpha\in\tau_{m}^{i}}f_{\abs{\alpha},\eps_n}(\sigma^{\alpha}_N)\right\rangle \leq Mm^{r+1}\epsilon_{n}
\]
thus by Markov's inequality,
\[
\E\mu^{\tensor\infty}_{N}\left(\sum_{i=1}^{M}\sum_{\alpha\in\tau_{m}^{i}}f_{\abs{\alpha},\eps_n}(\sigma^{\alpha}_N)\geq\sqrt{\epsilon_{n}}\right)\leq Mm^{r+1}\sqrt{\epsilon_{n}}.
\]%ENDMARK
\end{proof}

\begin{lem}\label{lem:gab-N}
Fix $\alpha,\beta\in\cA_r$ and let $k(\alpha,\beta)=\abs{\alpha\wedge\beta}+1$ and 
$g^N_{\alpha,\beta,\epsilon}$ as above. 
Then for $M,m,\epsilon_{n},N$
as above, we have 
\[
\E\mu^{\tensor\infty}_{N}\left(\sum_{\alpha\nsim\beta,\alpha,\beta\in\tau_{m}^{i}}
g^N_{\alpha,\beta,\epsilon_{n}}(\bsig_N)
\leq\sqrt{\epsilon_{n}},\forall i\in[M]\right)\geq1-Mm^{2r+2}\sqrt{\epsilon_{n}}
\]
\end{lem}
\begin{proof}
This follows by the same argument as above after observing that
for a fixed $\alpha,\beta,\in\tau_{m}$ with $\alpha\nsim\beta$, if 
$k=k(\alpha,\beta)$, 
\begin{align*}
\left\langle g_{\alpha,\beta}(\bsig_N)\right\rangle  
& =\mu_N^{\tensor4}\left(R_{12}\geq q_{k}+\epsilon,R_{13}\geq q_{k},R_{24}\geq q_{k},R_{23}< q_{k},R_{14}< q_{k}\right)\\
& \leq\mu_N^{\tensor3}\left(R_{12}\geq q_{k}+\eps, R_{13}\geq q_k, R_{23}<q_k\right),
\end{align*}
and that $\abs{\tau_m}\leq m^{r+1}$.
\end{proof}
\subsection{Approximate Ultrametricity}
\begin{lem}\label{lem:clust-1}
$\forall\eta,\epsilon,\delta$, there is an $m,M,n_{0}$ and a sequence $\tilde{N}(n;\eta,\eps,\delta)$ such that for
$n\geq n_{0}$ and $N\geq \tilde{N}(n)$, 
\[
\E\mu^{\tensor\infty}_{N}\left(\left(\bigcup_{i=1}^{M}E_{\tau^i_{m},\eps,\delta}(\mu_N)\right)\bigcap_{i=1}^{M}
\left\{ \sum_{\alpha\in\tau_m^i} f_{\abs{\alpha},\epsilon_{n}}(\sigma_N^{\alpha})\leq\sqrt{\eps_{n}}\right\} \cap
\left\{ \sum_{\alpha\nsim\beta\in\tau_m^i} g^N_{\alpha,\beta,\eps_{n}}(\bsig_N)\leq\sqrt{\eps_{n}}\right\} \right)
\geq1-\eta
\]
 \end{lem}
\begin{proof}
Re-write this as 
\[
\E\mu^{\tensor\infty}_{N}\left(A\cap B\cap C\right).
\]
By \prettyref{prop:regularity} and disintegration, we know that 
there is an $N_0(\eta,\eps,\delta)$ such that for $N\geq N_0$,
\[
\E\mu^{\tensor\infty}_{N}\left(A\right)\geq1-\eta/2.
\]
By \prettyref{lem:fk-N} and \prettyref{lem:gab-N}, if we let 
$\tilde{N}=N_1(n)\vee N_0(\eta,\eps,\delta)$, then for $N\geq \tilde{N}$,
\begin{equation}\label{eq:clust-1-abc}
\E \mu^{\tensor\infty}_N(A\cap B\cap C) \geq 1-\eta/2-M(m^{2r+2}+m^{r+1})\sqrt{\epsilon_n}.
\end{equation}
by inclusion-exclusion arguments. If we pick 
$n_0 \geq \lceil2\lg\left(2M(m^{2r+2}+m^{r+1})/\eta\right)\rceil,$
then for $n\geq n_0$ and $N\geq \tilde{N}(n;\eta,\eps,\delta)$, 
\[
\E \mu_N^{\tensor\infty}(A\cap B\cap C) \geq 1-\eta.
\]
\end{proof}

\subsubsection{ Proof of \prettyref{thm:main-result-1}}
We now begin the of \prettyref{thm:main-result-1}
\begin{thm}\label{thm:main-result-1-proof}
Let $\mu_N$ be as in \prettyref{thm:main-result-1}. Then for any $\zeta$-admissible sequence, there
are  sequences $a_N,b_N,\eps_N$ going to zero and $m_N\rightarrow \infty$ such that with probability tending to 1, 
there is a random collection of sets $\{C_{\alpha,N}\}_{\alpha\in\tau_{m_N}}$ that are an $(\eps_n,0)$ hierarchical exhaustion 
that is $(a_N,b_N)$-hierarchically clustering. 
\end{thm}
\begin{proof}
We begin by fixing a $\zeta$-admissible $\{q_k\}$ and $\nu\in\N$. Let  $\eta_\nu=\epsilon_\nu=\delta_\nu=\frac{1}{2^{\nu}}$.
By \prettyref{lem:clust-1}, for any $\nu$, we know that 
there are $m_\nu$,  $M_\nu$ ,  $n_{0}(\nu)$, and a sequence $\tilde{N}(n;\nu)$ such that if we consider the set
\[
E^N_{\nu}=\bigcup_{i=1}^{M_{\nu}}\tilde{E}_{\tau_{m_{\nu}}^{i},\epsilon_{\nu},\delta_{\nu}}(\mu_N)\cap
\left\{ \sum_{\alpha\in\tau_{m_{\nu}}}f^N_{\abs{\alpha},\eps_{n}}\left(\sigma_N^{\alpha}\right)\leq\sqrt{\eps_{n}}\right\} 
\cap\left\{ \sum_{\alpha\nsim\beta}g^N_{\alpha,\beta,\eps_{n}}(\bsig_N)\leq\sqrt{\epsilon_{n}}\right\} 
\]
then for $n\geq n_0$ and  $N\geq \tilde{N}(n;\nu)$, 
\[
\E\mu^{\tensor\infty}_{N}\left(E^N_\nu\right)\geq1-\frac{1}{2^{\nu}}.
\]
We can chose $m_{\nu},M_{\nu}$, and similarly $n_0(\nu)$
such that they all tend to infinity as $\nu\rightarrow\infty$. (To see this, mimic the following argument.) Let $N(1)= \tilde{N}(n_0(1))$ and 
$N(\nu+1)=\tilde{N}(n_0(\nu+1);\nu)\vee (N(\nu)+1)$. 
By definition 
\[
N(\nu) \geq \frac{\nu(\nu-1)}{2} +N(1)
\]
so if we let $\nu(N)=\sup\{\nu: N\geq N(\nu)\}$ we see that $\nu(N)\rightarrow\infty$ as well.
We can then suppress the dependence of $E_\nu^N$ on $\nu$.  Similarly, we replace the dependence
of all of the above in $\nu$ by its dependence on $N$. In particular, let $E^N=E^N_{\nu(N)}$, then 
\begin{equation}\label{eq:EN-lb}
\E\mu^{\tensor\infty}_N(E^N)\geq 1-\frac{1}{2^{\nu(N)}}.
\end{equation}
Thus as $N\rightarrow\infty$, this probability tends to one.

We now define the relevant sets. On $E^N$, we let $i=\min\{j: \bftW(\bsig,\mu_N)\in A_{\tau_{m(N)}^j,\eps_{\nu(N)},\delta_{\nu(N)}}\}$
and on $(E^N)^c$ set $i=-1$.  Then on $E^N$,  the $(B_\alpha)_{\alpha,\in\tau_{m(N)}^i}$
 form an $(\eps_{\nu(N)},\eps_{\nu(N)}/\abs{\tau_{m(N)}^i}^2)$-hierarchical
exhaustion by definition of the $\tilde{E}$'s from \prettyref{sec:outline}. By the same argument as in \prettyref{cor:exhaustion-2}, we see that the sets
\[
C_\alpha =B_\alpha\setminus \bigcup_{\beta\nsim\alpha} B_\beta
\]
form a $(2\eps_{\nu(N)},0)$-hierarchical exhaustion. This gives us our exhaustion with $\eps_N=\eps_{\nu(N)}$,
and $m(N)$. It remains to find the $a_N$ and  $b_N$, and to show \prettyref{eq:unif-close}
and \prettyref{eq:unif-far}. 

These follow from the definition of $f_{k,\eps}$ and $g_{\alpha,\beta,\eps}$. 
In particular, for the $C_\alpha$, we see that
\begin{align*}
\mu^{\tensor 2}_N\left(\sigma^1,\sigma^2\in C_\alpha:(\sigma^1,\sigma^2)\leq q_{\abs{\alpha}}-\eps_{n_0(N)}\right)
&\leq \mu^{\tensor 2}_N\left(\sigma^1,\sigma^2\in B_\alpha:(\sigma^1,\sigma^2)\leq q_{\abs{\alpha}}-\eps_{n_0(N)}\right)\\
&\leq \sum f^N_{\abs{\alpha},\eps_n}(\sigma_N^\alpha)\leq \sqrt{\eps_{n_0(N)}}\\
\end{align*}
and 
\begin{align*}
\mu^{\tensor 2}_N\left(\sigma^\alpha\in C_\alpha,\sigma^\beta\in C_\beta:(\sigma^\alpha,\sigma^\beta)\geq q_{k+1}+\eps_{n_0(N)}\right)
&\leq \mu^{\tensor 2}_N\left(\sigma^\alpha\in B_\alpha,\sigma^\beta\in B_\beta:(\sigma^\alpha,\sigma^\beta)\geq q_{k+1}+\eps_{n_0(N)}\right)\\
&\leq\sum_{\substack{\alpha,\beta\in\tau_m\\\alpha\nsim \beta}} g^N_{\alpha,\beta,\eps_{n_0(N)}}(\bsig_N)\leq \sqrt{\eps_{n_0(N)}}.
\end{align*}
Setting $a_N=\eps_{n_0(N)}$ and $b_N=\sqrt{\eps_{n_0(N)}}$ then gives us \prettyref{eq:unif-close}
and \prettyref{eq:unif-far}. 
\end{proof}

\prettyref{thm:main-result-1} is then a restatement of this result.

\subsection{Regularity Properties of Clusters\label{sec:clust-reg}}
In this section, we prove the convergence of the weights of clusters at a fixed depth on $\tau_m$ to 
a Poisson-Dirichlet process.

On the event $E^N$, we define $i=\min\{j: \bftW(\bsig,\mu_N)\in A_{\tau_{m_N}^j,\eps_{\nu(N)},\delta_{\nu(N)}}\}$ as in \prettyref{thm:main-result-1-proof}
and
\[
\tilde{Y}^N_{\alpha-i}=\begin{cases}
C_{\alpha} & \alpha\in\tau_{m}^{i}\\
\emptyset & else
\end{cases}
\]
where by $\alpha-i$  we mean the vertex we get by subtracting $i\cdot m$ from the first coordinate of $\alpha$.
On the event $(E^N)^{c}$ set $i=-1$ and 
$\tilde{Y}^N_{\alpha}=\emptyset,$
for all $\alpha$ in $\cA_r$. Finally, let 
\begin{equation}\label{eq:def-Y}
Y^N_{\alpha}=\mu_N(\tilde{Y}^N_{\alpha}).
\end{equation}
This gives us a sequence $Y_N\in[0,1]^{\cA_r}$.
Our goal is to show that if we consider the weights corresponding to those $\alpha\in\partial\tau_m$
and look at $(v_n^N)$, their decreasing rearrangement, this sequence of random variables
converges in distribution to a Poisson-Dirichlet process when considered as elements of the space of mass-partitions
\[
\cP_m=\left\{ (v_n) \in \R^\N: \sum v_i \leq 1,v_1\geq v_2\geq\ldots\geq0\right\}
\]
One equips this space with the subspace topology induced by the product topology on $\R^\N$. For more on this
space see \cite{AldExch83,Bert06}.

The proof of this result follows from an application of Talagrand's Identities after making
the following observation: due to the ultrametric nature of the $C_\alpha$, the indicator function of 
the event that two draws land in the same $C_\alpha$ is well approximated
by a function of the overlap. This makes precise the idea that we do a decomposition
into pure states with the $C_\alpha$.

We begin by proving the aforementioned observation. Let 
\[
U^N(\sigma^1,\sigma^2) =\indicator{\{\exists\alpha\in\N^r:\sigma^1,\sigma^2\in \tilde{Y}^N_\alpha\}}=\indicator{\text{both land in the same }\tilde{Y}^N_{\alpha}}.
\]
We denote this by $U^N_{12}$ as well. 
\begin{lem}\label{lem:indicator-approximation}
Let $\phi_{\kappa}(x)$ be the piece-wise linear function that
is $0$ on $[0,q_{r}-\kappa)$ and 1 on $[q_{r},1]$, where $\max\{\abs{q_k}-\abs{q_{k-1}}\}>\kappa>0$. 
Then
\[
\limsup_{N\rightarrow\infty}\E\gibbs{\abs{ U^N_{12}-\phi_{\kappa}(R_{12})}} _{\mu_{N}}\leq 2\zeta[q_r-\kappa,q_r+\kappa].
\]
\end{lem}
\begin{proof}
We leave out the dependence on $\mu_N$  in the Gibbs expectations and $N$ in the overlap for readability. Let 
\[
\Delta=\abs{U_{12}^N-\phi_\kappa(R_{12})}
\]
and let $L_N$ denote the event that both $\sigma^1$ and $\sigma^2$ land in $\cup_{\abs{\alpha}=r}Y_{\alpha}$. Then
\[
\E\langle\Delta\rangle = \E \gibbs{\Delta(\indicator{(E^N)^c} + \indicator{E^N,L_N^c}+\indicator{E^N,L_N})}
					= \eta(N)+ 2\eps_{\nu(N)}+I
\]
by \prettyref{thm:main-result-1-proof}, where the first comes from \prettyref{eq:EN-lb} and the second from the fact that $\tilde{Y^N_\alpha}$
are an $(2\eps_{\nu(N)},0)$-hierarchical exhaustion. We break up $I$  by noting that 
\begin{align*}
I &= \E\gibbs{\Delta\indicator{E^N,L_N}U_{12}\indicator{R_{12}\geq q_r}} 
+ \E\gibbs{\Delta\indicator{E^N,L_N}U_{12}\indicator{R_{12}<q_r}}+ \E\gibbs{\Delta\indicator{E^N,L_N}(1-U_{12})\indicator{R_{12}\geq q_r-\kappa}}\\
&+ \E\gibbs{\Delta\indicator{E^N,L_N}(1-U_{12})\indicator{R_{12}<q_r-\kappa}}
= II + III+ IV+ V.
\end{align*}
By definition of $\phi_\kappa$ and $U_{12}$,
$II=V=0.$

Furthermore, for $N$ large enough, 
\begin{align*}
IV/2 &\leq  \E\gibbs{ \indicator{E^N,L_N}  \left(\sum_{\substack{\alpha\nsim\beta\\ \alpha,\beta\in \partial\tau_{m(N)}}}
\indicator{\sigma^1\in\tilde{Y}^N_\alpha,\sigma^2\in\tilde{Y}^N_\beta}\right)\indicator{R_{12}\geq q_r -\kappa}}\\
&\leq \E\gibbs{\indicator{E^N,L_N}\left(
 \sum_{\substack{\alpha\nsim\beta\\ \alpha,\beta\in \partial\tau_{m(N)}}}
\indicator{\sigma^1\in\tilde{Y}^N_\alpha,\sigma^2\in\tilde{Y}^N_\beta}\right)\left[\indicator{R_{12}\geq q_r +\eps_{n_0(N)}}+\indicator{R_{12}\in[q_r-\kappa,q_r+\kappa]}\right]}\\
& \leq \E\gibbs{\indicator{E^N,L_N}\left(\sum_{\alpha\nsim\beta}g_{\alpha,\beta,\eps_{n_0(N)}}^N(\bsig)\right)} + \zeta_N[q_r-\kappa,q_r+\kappa]\leq \sqrt{\eps_{n_0(N)}}+\zeta_N[q_r-\kappa,q_r+\kappa]. 
\end{align*}
The second inequality comes from taking $N$ large enough that $\eps_{n_0(N)}<\kappa$, 
breaking up the interval $[q_r-\kappa,1]$ as
$[q_r-\kappa,q_r+\kappa]\cup[q_r+\eps_{n_0(N)},1]$, and using the union bound.
The third inequality follows from the definition of $g$ in \prettyref{eq:g-def}
and set inclusion. The fourth inequality follows from the definition of $E^N$.
It remains to study $III$. 

We begin again by breaking up the event.
\[
III  \leq\E\gibbs{\indicator{E^N}U_{12}\indicator{R_{12}<q_r-\eps_{n_0}}}
+\E\gibbs{\Delta U_{12}\indicator{R_{12}\in[q_r-\eps_{n_0},q_r)}}
= (i)+(ii).
\]
Observe that 
\[
(i)  =2\E\gibbs{ \sum_{\alpha\in\partial\tau_{m(N)}}f_{r,\epsilon_{n_0(N)}}(\sigma^{\alpha})\indicator{E^N}} \leq2\sqrt{\epsilon_{n_0(N)}}.
\]
On the interval $[q_r-\eps_{n_0},q_r)$, 
\[
\phi_\kappa(x)\geq 1-\frac{\eps_{n_0(N)}}{\kappa}
\text{ so that }
(ii)\leq \frac{\eps_{n_0(N)}}{\kappa}.
\]
Combining these results then gives 
\[
\E\gibbs{\abs{U_{12}-\phi_\kappa(R_{12})}}\leq \eta(N)+2(\eps_{\nu(N)}+\sqrt{\eps_{n_0(N)}}+\zeta_N[q_r-\kappa,q_r+\kappa])
+ 2\sqrt{\eps_{n_0(N)}}+\frac{\eps_{n_0(N)}}{\kappa}
\]
Taking limit superiors gives the desired result.
\end{proof}

Fix an $n$, an $s$, and a partition $\sqcup_{k=1}^{s}I_k=[n]$. Let 
\begin{equation}
F^N(\sigma^1,\ldots,\sigma^n)
=\indicator{\{\forall k\in[s],\exists\alpha\in\N^r:\forall i\in I_k, \sigma^i\in\tilde{Y}^N_\alpha\}}=\prod_{k=1}^s\prod_{i,j\in I_k}U_{ij}^N.
\end{equation}
In words, we partition the set $[n]$ into groups and ask if the groups are lying in the same $\tilde{Y}_\alpha^N$. Define $F^N_\kappa$ similarly using $\phi_\kappa$. The above result then has the following corollary.
\begin{cor}\label{cor:ind-approx-contd}
Fix $n$,$s$, and $I_k$  as above, and $\kappa>0$. Then
\[
\limsup \E\gibbs{\abs{F^N-F_\kappa}} \leq 2n(n-1)\zeta[q_r-\kappa,q_r+\kappa].
\]
\end{cor}
The proof of the above  follows by induction and the fact that $U_{ij}$ and $\phi_{\kappa}$ are bounded by $1$, so we omit it.
The rest is standard. First we find the parameter for the process.
\begin{lem}
Let $U^N_{12}$ as above, then 
\[
\lim_{N\rightarrow\infty}\E\left\langle U^N_{12}\right\rangle _{N}=\zeta[q_{r},1]
\]
\end{lem}
\begin{proof}
Note that 
\[
\abs{\E\left\langle U^N_{12}\right\rangle _{N}-\zeta[q_{r},1]}\leq\abs{\E\left\langle U^N_{12}-\phi_{\kappa}\right\rangle _{N}}+\abs{\E\left\langle \phi_{\kappa}\right\rangle _{N}-\E\left\langle \phi_{\kappa}\right\rangle _{\infty}}+\abs{\E\left\langle \phi_{\kappa}\right\rangle _{\infty}-\zeta[q_{r},1]}
= I+II+III.
\]
By convergence of the law of the overlap, 
\[
I \leq \E\gibbs{\abs{U_{12}^N -\phi_\kappa}}\text{, }
II \rightarrow 0
\text{, and }
III\leq \zeta[q_r-\kappa,q_r).
\]
If we combine these, we see that by \prettyref{lem:indicator-approximation}, 
\[
\limsup \abs{\E\gibbs{U^N_{12}}_{N}-\zeta[q_{r},1)}\leq \zeta[q_r-\kappa,q_r)+2\zeta[q_r-\kappa,q_r+\kappa].
\]
The result then follows after sending $\kappa\rightarrow 0$. 
\end{proof}
Then we show that the Ghirlanda-Guerra Identities hold for the relevant functions.
\begin{lem}\label{lem:GGI}
Fix $n$,$s$ and $I_k$ as above. Then we have the Ghirlanda-Guerra
identities
\begin{equation}\label{eq:GGI}
\lim\abs{n\E\gibbs{U^N_{1,n+1}F^N} -\E\gibbs{U^N_{12}}\E\gibbs{F^N} 
-\sum_{k=2}^{n}\E\gibbs{F^N U^N_{1,k}}}=0.
\end{equation}
\end{lem}
\begin{proof}
Note that by the Approximate Ghirlanda-Guerra Identities,
\[
\lim\abs{n\E\gibbs{\phi_{\kappa}(R_{1n+1})F_\kappa} -\E\gibbs{\phi_{\kappa}(R_{12})}\E\gibbs{F_\kappa}
 -\sum_{k=2}^{n}\E\gibbs{F_\kappa \phi_{\kappa}(R_{1k})}}=0
\]
for every $\kappa$ positive. Using boundedness of the relevant functions, \prettyref{lem:indicator-approximation},
and \prettyref{cor:ind-approx-contd}, we see that by a standard approximation argument
\[
\limsup\abs{n\E\gibbs{ U^N_{1,n+1}F^N} -\E\gibbs{U^N_{12}} \E\gibbs{F^N} 
-\sum_{k=1}^{n}\E\gibbs{F^NU^N_{1,k}}}\leq
C(n)\zeta[q_{r}-\kappa,q_{r}+\kappa].
\]
where $C(n)$ is finite and depends only on $n$.
Sending $\kappa\rightarrow 0$ gives the result since $q_{r}$ is a
$\zeta$ -continuity point. 
\end{proof}
We end this section by  showing the convergence to the Poisson-Dirichlet weights
\begin{thm}\label{thm:clust-reg-pd}
The weights $(v_{n}^{N})$ satisfy Talagrand's identities in the limit.
In particular they converge in distribution to $PD(\zeta[0,q_{r}])$
\end{thm}

\begin{proof}
Let 
\[
S_{N}(n_{1},\ldots, n_{s}) =\E\prod_{k\leq s}\sum v_{n}^{n_{k}} = \E\left\langle F^N\right\rangle. 
\]
Note that 
\[
S_{N}(n_{1}+1,\ldots,n_{s})=\E\left\langle U^N_{1,n+1}F^N\right\rangle 
\text{ and }
S_{N}(n_{2},\ldots,n_{k}+n_{1},\ldots,n_{s})=\E\left\langle U^N_{1k}F^N\right\rangle 
\]
so that by the Ghirlanda-Guerra Identities \prettyref{eq:GGI}, 
\begin{align*}
nS_{N}(n_{1}+1,\ldots,n_{s}) &=S_{N}(2)S_{N}(n_{1},\ldots,n_{s})+(n_{1}-1)S_{N}(n_{1},\ldots,n_{s})\\
&+\sum_{2\leq k\leq s}n_{k}S_{N}(n_{2},\ldots,n_{k}+n_{1},\ldots,n_{s})+o(1).
\end{align*}
For $k\geq2$ , we know that on $\cP_m$, the polynomials
$p_k((v_n))=\sum_n v_{n}^{k}$
are continuous (bounded) functions in $(v_{n})$ \cite{PanchSKBook}. 
If we then pass to a weakly convergent subsequence of $(v^N_n)$'s, 
all of the $S_N(\ldots)$'s will converge.
For any such limit point, we then have  that the Talagrand Identities
with parameter $\theta=\zeta[0,q_r)$ hold exactly. The latter uniquely specifies
the limit point as $PD(\zeta[0,q_r))$ \cite{TalBK03, PanchSKBook}.
Thus, by the subsequence principle
we know that the sequence properly converges.
\end{proof}
Note that by simply forgetting the last $r-k$ overlap values, we have the following corollary. Alternatively, one could
repeat the above proofs modifying as necessary.
\begin{cor}
The same is true if one rearranges the vertices at any fixed depth $k$ except the parameter for the Poisson-Dirichlet
Process becomes $\zeta[0,q_k)$.
\end{cor}
\section{Convergence to Ruelle Cascades}\label{sec:cascades}
In this section, we improve on the above, by showing that the collection of all of the weights above, once correctly
rearranged, form a Ruelle Probability Cascade in the limit.

Let $Y^N_\alpha$ be as before, and let $\bv^N=(v^N_\alpha)$ be these weights placed in standard order. (We point out here that $v_\alpha=0$ is possible.) Note that
these weights can be thought of as random variables in the space of cascades of depth $r$.

\begin{defn}\label{defn:cascade-space}
The space of \emph{Cascades of depth} $r$ is the space
\[
\cC_r = \{(w_\alpha)_{\alpha\in\cA_r}\in [0,1]^{\cA_r}: w_{\alpha 1}\geq w_{\alpha 2}\geq \ldots\geq 0;
\sum_{\abs{\alpha}=k} w_\alpha \leq 1,\forall k\in[r];w_\alpha\geq\sum_{\beta\in child(\alpha)}w_\beta\}
\]
 which is topologized as a subspace of the product space $[0,1]^{\cA_r}$. 
 A \emph{cascade} is an element of this space. 
 A cascade is said to be \emph{proper} or a \emph{probability cascade} if the
 inequalities relating to the sums are all equalities. Otherwise the cascade is called \emph{improper}.
 \end{defn}
 
For ease of notation we omit the depth when talking about the space of cascades when it is not ambiguous. 
The space of cascades is compact and Polish. 
Notice that in 
this language, an RPC is a random variable in the space of cascades that is almost surely proper. We note
that this space of cascades is different from that defined in \cite{Rue87}, though  one can go from a 
cascade in the sense of Ruelle to a cascade in the above sense.

Our goal is to show that the above sequence of weights converges in distribution to a Ruelle Probability
Cascade. The proof follows from an application of the uniqueness portion of the Dovbysh-Sudakov theorem.
In particular, we encode the cascades into ROSts and use this to uniquely identify properties
of the limit.  As before the key observation is that the relative locations of points in the cascade is well
approximated by their overlaps. 
The proof is in two parts. First we describe a map taking a cascade to a
ROSt and demonstrate that if $\bv$ is a limit point of the sequence $\bv^N$, then for any subsequence that 
converges to it, the corresponding sequence of ROSts converges to that of $\bv$.
We  then show that for the sequence $\bv^N$, the corresponding sequence of ROSts has only one limit point
which is an RPC. This combined with the previous result and the uniqueness portion of Dovbysh-Sudakov 
uniquely identifies all of the limit points as being the same RPC, so that we can conclude the result by the
subsequence principle.

\subsection{Convergence to the ROSt of the limit point}
The first part of the proof requires studying an encoding of the above distributions into a ROSt. 
For any cascade $\bv\in\cC_r$, let 
\[
v_\partial =1-\sum_{\abs{\alpha}=r} v_\alpha.
\]
We think of $\partial$ as the dustbin where we place the dust from the Ruelle Probability Cascade 
(see \cite{Bert06} for the corresponding terminology for mass partitions). 

We set up the encoding as follows. Let $\{e_\alpha\}_{\alpha\in(\cA_r\cup\{\partial\})\setminus\{\emptyset\}}$ be orthonormal basis vectors 
for $\ell_2$. Consider the vectors 
\begin{equation}\label{eq:h_alpha}
h_\alpha = \sum_{\beta\precsim\alpha}(q_{\abs{\beta}} - q_{\abs{\beta}-1})^{\frac{1}{2}}e_\beta
\end{equation}
for $\alpha\in\cA_r$ and $h_\partial=\sqrt{q_r}e_\partial$, where $q_0=0$. For any $\bv\in\cC_r$ and sequence $\{q_k\}_{k=1}^r$, we
define the map $\cR:\cC_r\times[0,1]^r\rightarrow\Pr B_{\ell_2}(0,1)$ by 
\begin{equation}\label{eq:rost-map}
\cR(\mathbf{v},\{q_k\};h_\alpha)=\begin{cases}v_\alpha & \alpha\neq \partial\\
v_\partial & \alpha=\partial.
\end{cases}
\end{equation}
In the following we suppress the dependence on the sequence $\{q_k\}$ when it is 
unambiguous. Note that on the space of proper cascades this map is injective. It then
extends as a map from the space of probability measures on proper cascades to the space of
laws of ROSts.

Take $\bv^N$ as above and $\bv$ a limit point such that $\bv^{N_k}\convdist \bv$ for some subsequence $\{N_k\}$. 
We denote the corresponding Gram-DeFinetti laws of the associated ROSt by $\cL_N$ and $\cL$. Our goal will be to show that these sequences converge weakly. We begin by showing convergence of the dustbins.
\begin{lem}
Let $S^{(N)}_d = \sum_{\abs{\alpha}=d} v^N_{\alpha}$ and $S_d = \sum_{\abs{\alpha}=d} v_{\alpha}$ where 
$(v_{\alpha})$ is a distributional limit point of the sequence $(v^N_\alpha)$. Then if $N_l$ is a subsequence along which 
this convergence happens,
$S^{N_l}_d\convdist S_d.$
for every $d\in [r]$.Furthermore this convergence can be taken to happen simultaneously with the convergence of the $(v^{N_l}_\alpha)$.
\end{lem}
\begin{proof}
We work with $d=r$. The proof for the rest of the result follows \emph{mutatus mutandis}. We also pass to the subsequence
immediately to avoid cumbersome notation. Let 
\[
\eta_N^k = \sum_{\alpha\in \partial\tau_k} v_\alpha^N
\text{ and } 
\eta^k = \sum_{\alpha\in \partial\tau_k} v_\alpha.
\]
Notice that $\eta_N^k \convdist \eta^k$ and that $\eta^k \convdist S_r$. 
By an approximation theorem (see \cite{Kal97}), it suffices to show that 
\[
\lim_{k\rightarrow\infty}\limsup_{N\rightarrow\infty} \E \left(\abs{\eta_N^k - S_r^N}\right)=0.
\]

To see this, fix $\tau_k$. Let $\tau_{k,l}$ be the $(\N,\N,\ldots,\N,k,\ldots ,k)$-shaped tree where there are $l$ $\N$'s in the 
$r$-tuple. Note that $\tau_{k,0}=\tau_k$ and $\tau_{k,r}=\cA_r$
For any subset $\tau^\prime \subset \cA_r$ , let 
\[
S^{(N)}(\tau^\prime) = \sum_{\alpha\in\tau^\prime} v^N_\alpha
\]
and similarly for $S(\tau^\prime)$ in the limit. Note that
\[
\abs{S^N_r-\eta^k_N}\leq \sum_{l=0}^{r-1}\abs{S^{(N)}(\partial\tau_{k,l+1})-S^{(N)}(\partial\tau_{k,l})},
\]
so it suffices to show that each summand vanishes upon taking expectations and limits. To see this, consider the difference
\[
\abs{S^{(N)}(\partial\tau_{k,l+1})-S^{(N)}(\partial\tau_{k,l})}=S(\partial\tau_{k,l+1}\setminus\partial\tau_{k,l}) \leq S(\partial(\cA_r\cap\{\alpha_{l+1}>k\})).
\]
Since the sum of the children is less than that of their
parents at every level by definition of $\cC_r$, this sum is at most
\[
RHS\leq \sum_{\abs{\alpha}=l,n\geq k+1 }v^{(N)}_{\alpha n} = S(\partial(\cA_{l+1}\cap\{\alpha_{l+1}>k\})).
\]
Let $\tilde{v}^{l+1,N}_{m}$ be the decreasing rearrangement of $(v^N_\alpha)_{\abs{\alpha}=l+1}$. 
Then
\[
RHS\leq \sum_{m\geq k+1} \tilde{v}_m^{l+1,N}
\]
since $(v_\alpha)$ are in standard order and thus  the previous
sum must be missing the contribution of the first $k$ of each family at depth $l+1$.
Consequently,
\[
\E\abs{S^{(N)}(\tau_{k,l+1})-S^{(N)}(\tau_{k,l})}\leq \eps + \prob(\sum_{m\geq k+1}\tilde{v}_m^{l+1,N}\geq \eps/2).
\]
For each $l$, we have 
$(\tilde{v}^{l,N}_m)\convdist (\tilde{v}^l_m)$
where $(\tilde{v}_m^l)\sim PD(\zeta[0,q_k])$ by \prettyref{thm:clust-reg-pd}. This means that 
\[
\limsup_{N\rightarrow\infty}\prob(\sum_{m\geq k+1}\tilde{v}_m^{l,N}\geq \eps/2) \leq \limsup_{N\rightarrow\infty} \prob(\sum_{m\leq k}\tilde{v}_m^{l,N}\leq 1-\eps/2)\leq \prob(\sum_{m\geq k+1}\tilde{v}_m^{l}\leq 1-\eps/2)\leq \eps
\]
for $k\geq k_0$ for some appropriately chosen $k_0$. Combining these results gives the result.

That the convergence happens simultaneously can be seen by using the fact that the relevant spaces are metrizable
and that adding the dustbin can be done by adding an extra factor of $[0,1]$ with the usual product metric. 
\end{proof}
\begin{rem}%%%MARK
We would like to point out here that as a consequence of the above convergence and the compactness of
$\cC_r$, we have that for $\bv$ as above,
\[
v_\alpha = \sum_{\beta\in child(\alpha)} v_\beta.
\]
This coupled with the injectivity of the map from proper cascades to ROSts will be used in the subsequent.
\end{rem}%%%END MARK
With this in hand we can now show the convergence of the overlap distributions.

\begin{lem}\label{lem:overlap-conv-lt-pt}
Let $\bv$ be a limit point of the sequence $\bv^N$ and let $\bv^{N_k}\rightarrow\bv$. It follows that
$\cL_{N_k}\convdist\cL$
\end{lem}
\begin{proof}
For ease of notation, we pass to the subsequence and eliminate the subscript $k$ in the above. Note that
\[
(v_\alpha^N,v_\partial^N,S_d^N)\convdist(v_\alpha,v_\partial,S_d).
\]
By Skorokhod's representation theorem, there is a
probability space $(\tilde{\Omega},\cG,\tilde{\prob})$ and random variables 
\[
(w_\alpha^n,w^N_\partial,W_d^N)\eqdist(v_\alpha^N,v_\partial^N,S_d^N)
\]
and similarly for $(w_\alpha,w_\partial,W_d)$ such that 
\[
(w_\alpha^n,w^N_\partial,W_r^N)\rightarrow(w_\alpha,w_\partial,W_d)
\]
$\tilde{\prob}$-almost surely. Note that by the distributional equality, for any $f\in C[-1,1]^{n^2}$, 
\[
\E\gibbs{f}_{\cR(v^N)}=\E\gibbs{f}_{\cR(w^N)}
\]
as the map 
$w^N \mapsto \gibbs{f}_{\cR(w^N)}$
is bounded and measurable (though not necessarily continuous). Thus it suffices to show
\[
\E\gibbs{f}_{\cR(w^N)}\rightarrow\E\gibbs{f}_{\cR(w)}.
\]
We begin by noting that 
\[
w^N_\partial +\sum w_\alpha^N =1=w_\partial +\sum w_\alpha,
\]
so that by Scheff\'e's lemma, we have the convergence 
\[
(w_\alpha^N,w_\partial^N)\rightarrow(w_\alpha,w_\partial)
\]
in $\ell_1(\partial\cA_r \cup\{\partial\})$ almost surely. As a result, 
\[
\gibbs{f}_{\cR(w^N)}=\sum_{\alpha_1,\ldots,\alpha_m\in\partial\cA_r\cup\{\partial\}}w_{\alpha_1}\ldots w_{\alpha_m} f\left((h_{\alpha_i},h_{\alpha_j})\right)
\]
converges to
$\gibbs{f}_{\cR(w)}$
almost surely by an $\eps/3$-type argument. The convergence of the means then follows by bounded convergence
theorem.
\end{proof}
\subsection{Convergence of the ROSts to an RPC}\label{sec:rost-conv-to-rpc}
We begin with the following approximation argument which is similar to \prettyref{lem:indicator-approximation}.
In the following we fix $q_0=0$ and $q_{r+1}=1$
\begin{lem}\label{lem:indicator-approximation-2}
Fix ${\talpha},\tbeta\in \N^{r+1}$ distinct. Let $k = \abs{\talpha\wedge\tbeta}$. Let
\[
A(\talpha,\tbeta)=\left\{\exists \alpha,\beta\in\N^r,\gamma\in\N^k:\gamma=\alpha\wedge\beta,
\sigma^\talpha\in \tilde{Y}_\alpha, \sigma^\tbeta\in \tilde{Y}_\beta\right\}
\]
and for each $\kappa>0$ such that $\kappa <\min_k\{\abs{q_k-q_{k+1}}\}/3$, set
\[
B_\kappa=\{R_{\talpha\tbeta} \in [q_k-\kappa,q_{k+1}+\kappa]\}.
\] 
where if $k=r$, the latter set ends at $q_{r+1}=1$, and if $k=0$ it begins at $q_0=0$.Then
\[
\limsup \E\gibbs{ \abs{\indicator{A(\talpha,\tbeta)}-\indicator{B_\kappa}}} \leq \zeta[q_k-\kappa,q_k+\kappa]\indicator{k>0}+\zeta[q_{k+1}-\kappa,q_{k+1}+\kappa]\indicator{k<r}.
\]
\end{lem}

\begin{proof}
We drop the subscript $\kappa$ and suppress the dependence of $A$ for readability. 
The proof follows by case analysis. First break up
\[
\E\gibbs{\abs{\indicator{A}-\indicator{B}}} \leq \E\gibbs{\abs{\ldots}\indicator{E^N}}
+\prob((E^N)^c)\leq  I +\frac{1}{2^{\nu(N)}}
\]
where the second inequality follows from the definition of $E^N$.

Let $L_N$ be the event that both $\sigma^\talpha$ and $\sigma^\tbeta$ are in 
$\cup_{\abs{\alpha}=r}\tilde{Y}_\alpha$. Break up $I$ as follows.
\[
I \leq 2\E\gibbs{\indicator{E^N,L_N}\left(\indicator{A\cap B^c}+\indicator{B\cap A^c}\right) + \indicator{E^N,L_N^c}}
\leq 2(II+III)+2\eps_{\nu(N)}.
\] 
On $B^c$ either
\[
R_{\talpha\tbeta}\geq q_{k+1}+\eps_{n_0(N)}
\text{ or }
R_{\talpha\tbeta}<q_k -\eps_{n_0(N)}
\]
for $N$ large enough, where the first case is impossible if $k=r$ and the second is impossible if $k=0$. Thus 
\[
II \leq \E\gibbs{\indicator{E^N,L_N\cap A}\indicator{R_{\talpha\tbeta}\geq q_{k+1}+\eps_{n_0(N)}}}\indicator{k<r}
+\E\gibbs{\indicator{E^N,L_N\cap A}\indicator{R_{\talpha\tbeta}<q_k -\eps_{n_0(N)}}}\indicator{k>0}
\leq 2\sqrt{\eps_{n_0}(N)},
\]
after integrating first in the ``tilded'' variables , where the second inequality comes from
the fact that both terms in on the RHS are at least $q_k$. 

We turn now to $III$. On $A^c$, we have that 
\[
\nexists \gamma:\abs{\gamma}=\abs{\talpha\wedge\tbeta} \text{ and } \sigma^\talpha,\sigma^\tbeta \in \tilde{Y}_\gamma
\]
On $A^c\cap L_N$, we also have an $\alpha,\beta$ and $\gamma$ such that $\gamma=\alpha\wedge\beta$
with $\sigma^\talpha\in\tilde{Y}_\alpha$, $\sigma^\tbeta\in\tilde{Y}_\beta$, but $\abs{\gamma}\neq k$.

We break up $III$ as follows
\[
III\leq \E\gibbs{\indicator{E^N\cap A^c\cap B\cap L_N}\left(\indicator{\abs{\gamma}<k-1}+
\indicator{\abs{\gamma}=k-1}+
\indicator{\abs{\gamma}=k+1}+\indicator{\abs{\gamma}>k+1}\right)} = (iii)+(iv)+(v)+(vi).
\]

If $\abs{\gamma}<k-1$,possible only if $k>1$, then on $A^c\cap B \cap L_N$ 
\[
R_{\talpha\tbeta}\geq q_k-\kappa\geq q_{k+1}-\eps_{n_0(N)}\geq q_{\abs{\gamma}+1}+\eps_{n_0(N)}
\]
for $N$ large enough so that
\[
(iii)\leq \indicator{k>1}\E\gibbs{\indicator{E^N}(\sum_{\alpha\nsim\beta} g_{\alpha,\beta,\eps_{n_0(N)}}(\bsig))}\leq \sqrt{\eps_{n_0(N)}}.
\] 
If$\abs{\gamma}>k+1$, possible only when $k < r$, 
\[
R_{\talpha\tbeta}\leq q_{k+1}+\kappa < q_{\abs{\gamma}} -\eps_{n_0(N)}
\]
which is controlled similarly by $f$ so 
\[
(vi)\leq \sqrt{\eps_{n_0}(N)}\indicator{k<r}.
\]

If $\abs{\gamma}=k+1$, again possible only if $k<r$, then either
\[
R_{\talpha\tbeta}\in[q_k-\kappa,q_{k+1}-\kappa]
\text{ or }
R_{\talpha\tbeta} \in[q_{k+1}-\kappa,q_{k+1}+\kappa]
\]
which are controlled by $f$ and bounded by $\zeta_N[q_{k+1}-\kappa,q_{k+1}+\kappa]$ respectively. Thus
\[
(v) \leq (\sqrt{\eps_{n_0(N)}} + \zeta_N[q_{k+1}-\kappa,q_{k+1}+\kappa])\indicator{k<r}
\]

Finally if $\abs{\gamma}=k-1$, again possible only if $k>0$, then either
\[
R_{\talpha\tbeta}\in[q_k+\kappa,q_{k+1}+\kappa]
\text{ or }
R_{\talpha\tbeta}\in[q_k-\kappa,q_k+\kappa]
\]
which are controlled by $g$ and bounded by $\zeta_N[q_k-\kappa,q_k+\kappa]$ respectively,
so
\[
(iv)\leq \sqrt{\eps_{n_0(N)}}+\zeta_N[q_k-\kappa,q_k+\kappa]\indicator{k>0}
\]

Combining all of these gives
\[
III\leq 4\sqrt{\eps_{n_0(N)}}+\zeta_N[q_k-\kappa,q_k+\kappa]\indicator{k>0}+\zeta_N[q_{k+1}-\kappa,q_{k+1}+\kappa]\indicator{k<r}
\]

Combine $II$ and $III$ and take limit superiors of both sides to arrive at
\[
\limsup\E\gibbs{\abs{\indicator{A}-\indicator{B}}}\leq \zeta[q_k-\kappa,q_k+\kappa]\indicator{k>0}+\zeta[q_{k+1}-\kappa,q_{k+1}+\kappa]\indicator{k<r}.
\]
\end{proof}

Note that since the $q_k$ are continuity points, we have the following, which follows by a standard approximation
argument.
\begin{cor}\label{cor:rpc-set-approx}
Set $\kappa=0$ in the definition of $B_\kappa$. Then 
\[
\limsup_{N\rightarrow0}\E\gibbs{\abs{\indicator{A}-\indicator{B_0}}}=0.
\]
\end{cor}
With this observation, 
we now show that any limit point of the overlap structure corresponding to the $v^N$'s 
must an RPC. To do this consider the following. Let the function $\Gamma$ be defined by
\begin{equation}\label{eq:Gamma}
\Gamma(q)=\begin{cases} q_{k} & q\in[q_k,q_{k+1})\\
q_r &q\geq q_r\\
0 & q\in [0,q_1), 
\end{cases}
\end{equation}
and define the \emph{$\Gamma$-approximator} to $\mu$ be the overlap distribution you get by
pushing the distribution of $\mu$ through by the map $R_{ij}\mapsto \Gamma(R_{ij})$. Recall from \prettyref{fact:RPC}
that the $\Gamma$-approximator to $\mu$ is given by a RPC with the jump weights $\zeta_k = \zeta[q_k,q_{k+1})$ (with $q_0=0$). 

Consider a matrix of possible overlap values $(q_{ij})_{ij\leq n}$. 
Corresponding to this matrix there is a tree $\tau$ whose 
structure mimics the ultrametric defined by this matrix. If $q_{ij}$ takes on the values
$\{q_k\}_{k=0}^{r}$, then this tree is of depth $r+1$. We generate such a tree as follows: take
the points $[n]$ as the leaves and create root leaf paths of length $r+1$ to each of them. If $i$ and $j$ are
such that $q_{ij}=q_k$, then we join their paths starting at the root and ending at depth $k$. We call
such a tree an \emph{encoding} of the overlap structure $(q_{ij})$. 

View $\cA_r\cup\{\partial\}$ as a tree by adding a root leaf path to $\cA_r$ with one vertex at each level ending in $\partial$.
For any finite rooted tree $\tau$ of depth $r+1$, we let
\[
h(\tau) =Emb(T(\tau;r),\cA_r\cup\{\partial\})
\]
be the set of embeddings of $T(\tau;r)$, the tree of depth $r$ obtained by deleting the leaves of $\tau$, 
into $\cA_r\cup\{\partial\}$,
and let 
\[
\cP(\tau)=\{n_{\alpha}\}_{\alpha\in\tau,\abs{\alpha}=r}
\]
where $n_{\alpha}=card(child(\alpha))$.
With this we can then prove the next step in the convergence result.

\begin{lem}\label{lem:overlap-conv-to-GM}
The ROSt corresponding to $\cR(\bv^N)$ converges in law to the $\Gamma$-approximator of $\mu$.
\end{lem}

\begin{proof}
Define $A(\talpha,\tbeta)$ as above making explicit the dependence on $\talpha$ and $\tbeta$. Let $\tau$ be the encoding of the overlap structure $(q_{ij})$. Notice that
\[
\gibbs{R^n = (q_{ij})}_{\cR(\bv^N)}= \sum_{\phi\in h(\tau)}\prod_{\alpha\in \partial T(\tau;r)}v^{n_\alpha}_{\phi(\alpha)}.
\]
It follows that
\[
\abs{\gibbs{R^n=(q_{ij})}_{\cR(\bv^N)}- \gibbs{\prod_{\talpha}\neq\tbeta\in\partial\tau} \indicator{A(\talpha,\tbeta)}\indicator{E^N}}_{\mu_N}
\leq \gibbs{(E^N)^c} + \gibbs{E^N\cap L_N^c}
\]
where $L_N=L_N(\tau)$ is again the event that all of the replica indexed by $\partial\tau$ lands inside of 
$\cup_{\abs{\alpha}=r} \tilde{Y}^N_\alpha$. The first term comes from the case where $\partial$ is in the
 image and gets all of the mass (i.e.  all of the $v_\alpha = 0$). The second case comes 
 from the chance that $\partial$ is in the image and gets some of the mass. 
Taking expectations then gives us
\[
\abs{\E\gibbs{R^n = (q_{ij})}_{\cR(\bv^N)}-\E\gibbs{\prod_{\alpha\neq\beta\in\partial\tau} \indicator{A(\alpha,\beta)}\indicator{E^N}}_{\mu_N}}\leq \prob((E^N)^c)+\prob(E^N\cap L_N^c)
\]
which is vanishing by \prettyref{thm:main-result-1-proof}. By \prettyref{cor:rpc-set-approx} it follows that
\[
\E\gibbs{\prod_{\alpha\neq\beta\in\partial\tau} \indicator{A(\alpha,\beta)}\indicator{E^N}}_{\mu_N}\rightarrow\E\gibbs{R_{ij}\in I_{ij}}_\mu
\]
where  if we let $k_{ij}$ be such that $q_{ij}=q_{k}$, $I_{ij}=[q_{ij},q_{k_{ij}+1}]$. Combining these gives the result.
\end{proof}
\subsubsection{Proof of \prettyref{thm:main-result-2}} 
We can then prove the main result.
\begin{thm}
The sequence $\bv^N$ converges to the weights of an RPC with weights $\zeta_k$ as in \prettyref{fact:RPC}.
\end{thm}
\begin{proof}
Suppose that $\bv$ is a limit point of the sequence $\bv^N$. Pass to a subsequence which converges to $\bv$.
By \prettyref{lem:overlap-conv-lt-pt}, we know that for this subsequence
\[
\cL(\cR(\bv^{N_k}))\rightarrow\cL(\cR(\bv)).
\]
By \prettyref{lem:overlap-conv-to-GM}, however, we know that the overlap distributions have a
unique limit point, namely the $\Gamma$-approximator to $\mu$. The  latter corresponds to an 
RPC, by \prettyref{fact:RPC}. 

Let $\nu$ be the Dovbysh-Sudakov measure corresponding to the RPC. 
By the above we know that, since its overlap distribution is the same as $\cL(\cR(\bv))$,
there is a coupling of $\cR(\bv)$ and $\nu$, and a random isometry $T$ from the closed linear span of the support of 
$\nu$ to that of $\cR(\bv)$ such that almost surely 
$T_*\nu=\cR(\bv)$.

In particular this means that there is a random bijection $\pi:\partial\cA_r\rightarrow\partial\cA_r$ such that
$Th_\alpha=h_{\pi(\alpha)}$ for $\alpha\in\partial\cA_r$. Our goal will be to show that $\pi=Id$.

We begin by noting that since $T$ is an isometry and there
are only finitely many possible distances, $\pi$ extends to all of $\cA_r$ in such a way that it preserves
the parent-child relationship, 
$\pi(\alpha)\wedge\pi(\beta)=\pi(\alpha\wedge\beta)$.

Let $W_\alpha$ be the weights of the cascade arranged in standard order. Observe that
\[
W_{\alpha} = T_* \nu(h_{\pi(\alpha)})=\cR(\bv;h_\pi(\alpha))=v_{\pi(\alpha)},
\]
for $\alpha\in\partial\cA_r$. Since $\pi$ preserves the parent-child relationship, and since $v_\alpha$ is the sum over
all children $\beta$  of $\alpha$ of the $v_{\beta}$ and similarly for $W_{\alpha}$, it follows that this equality
extends to the whole tree. The $W_n$ are almost surely distinct, therefore $v_n$ must be as well. 
As both are already in decreasing order, it follows that $W_n =v_n$. Continue this argument down the tree. 
This implies $\pi =Id$. Thus the two sets of weights are equal in law.
\end{proof}
\section{Talagrand's Orthogonal Structures Conjecture}\label{sec:OSC}
We study here the proof of the orthogonal structures conjecture. We begin with a few preliminary results about
Poisson-Dirichlet processes. These results, combined with approximate ultrametricity
 and the Talagrand Positivity Principle
will then establish the theorem. The proofs of the following are standard exercises
in weak convergence and Talagrand's Identities, so we omit them.

\begin{lem}\label{lem:ak-for-pd}
Fix $\theta\in(0,1)$, and let $\mathbf{v}=(v_{n})\sim PD(\theta)$.
Then for any $p\in(0,1)$ , there is a sequence of positive real numbers
$(a_{k})$ such that 
\[
\prob(v_{k}>a_{k},\forall k)\geq p
\]
\end{lem}

\begin{lem}\label{lem:theta-det-conv}
Let $\mathbf{v}_{n}\sim PD(\theta_{n})$. If there is a $\theta\in(0,1)$
such that $\theta_{n}\rightarrow\theta$, then 
$\mathbf{v}_{n}\convdist\mathbf{v}$
where $\mathbf{v} \sim PD(\theta)$.
\end{lem}
\begin{lem}\label{lem:ak-for-vn}
Take $q^{(n)}\rightarrow0$ continuity points of $\zeta$. Let $\theta_{n}=\zeta[0,q^{(n)})$
and let $\mathbf{v}_{n}\sim PD(\theta_{n})$. If $0$ is an atom of
$\zeta$, then for every $p\in(0,1)$ there is a sequence $(a_{k})$
of positive real numbers such that for every $\epsilon$ positive
and $k_{0}\in\N$, there is an $n_{0}\in\N$ such that for all $n\geq n_{0}$,
\[
\prob(v_{k}^{n}>a_{k}\forall k\leq k_{0})\geq p \text{ and } 
q^{(n)}<\epsilon.
\]
\end{lem}
\begin{lem}\label{lem:ak-for-yn}
For each $n$, let $Y^{n,N}$be defined as in \prettyref{eq:def-Y} for $\{q^{n}\}$.
Then for every $p\in(0,1)$, there is a sequence $(a_{k})$ such that
for every $\epsilon>0,k_{0}\in\N,$ there is an $n_{0}$ and an $N_{0}$
such that for $N\geq N_{0}$ 
\[
\prob(Y_{k}^{N,n_{0}}>a_{k}\forall k\leq k_{0})\geq p \text{ and } q^{n_0}<\eps
\]
\end{lem}
\begin{proof}
Recall that by \prettyref{thm:clust-reg-pd},  for each $n$ 
\[
Y^{n,N}\rightarrow Y^{n}
\]
weakly where $Y^{n}\sim PD(\zeta[0,q^n))$. It then follows from
\prettyref{lem:ak-for-vn} that for every $p\in(0,1)$ there is a sequence $(a_k)$ such
that for every $\epsilon,k_{0}$ as above, there is an
$n_{0}$ such that 
\[
\prob(Y_{k}^{n_{0}}>a_{k}\forall k\leq k_{0})\geq\frac{1+p}{2}
\]
Combining these results and using again the fact that $A_{k}$ as
defined in \prettyref{lem:ak-for-pd} is open then gives us that there is an $N_{0}$
such that for $N\geq N_{0}$, we have 
\[
\prob(Y_{k}^{n_{0,}N}>a_{k,}\forall k\leq k_{0})\geq p.
\]
\end{proof}

\begin{lem}\label{lem:tal-pos}
For every $\epsilon,\eta$ positive, there is an $\tilde{N}$ such
that for $N\geq\tilde{N}$ 
\[
\prob(\mu_{N}^{\tensor\infty}(R_{12}<-\epsilon)>\epsilon/3)<\eta.
\]
\end{lem}
The proof follows from \prettyref{prop:GGI-cons} as before so we omit it.

\begin{thm}[Orthogonal Structures]\label{thm:low-temp-OSC}
There is a sequence $(\tilde{a}_{k})$ of positive numbers such that
for every $\epsilon$ positive and $k_{0}\in\N$, there is an $N_{*}$
such that for $N\geq N_{*}$, with probability at least 3/4, 
\[
\exists\{A_{k}^{N}\}_{k\leq k_0}\subset\Sigma_{N}:\mu_{N}(A_{k}^{N})\geq\tilde{a}_{k}
\]
 and 
\[
\forall k,l\leq k_{0},k\neq l,\left\langle \abs{R_{12}}\indicator{\sigma^{1}\in A_{k},\sigma^{2}\in A_{l}}\right\rangle \leq\epsilon.
\]
\end{thm}
\begin{proof}
Using \prettyref{lem:ak-for-yn}, we know that for $p=7/8$, we have a sequence $(\tilde{a}_{k})$
such that for every $\epsilon$ positive and $k_{0}\in\N$, there
is an $N_{0}$ and $n_1$ so that for $N\geq N_{0}$,
\[
\prob(Y_{k}^{N,n_1}>\tilde{a}_{k},\forall k\leq k_{0})>7/8
\text{ and }  q^{n_1}<\epsilon/6.
\]
Let $N_{1}$ be such that $m_{N}\geq k_{0}$, $b_{N}(n_1)<\epsilon/3$, and $a_{N}(n_1)<\eps/6$ 
where $m_N$, $b_{N}$, and $a_N$ are from \prettyref{thm:main-result-1-proof}.
Finally take $N_{2}=\tilde{N}$ from \prettyref{lem:tal-pos} for $\eta=1/8$
so that 
\[
\prob(\mu_{N}^{\tensor\infty}(R_{12}<-\epsilon)>\epsilon/3)<\frac{1}{8}
\]
for $N\geq N_{2}$. Set $N^{*}= \max N_{i}$. Then for all $N\geq N^*$,
\begin{equation}\label{eq:prob-34}
\prob(Y_{k}^{N,n_1}>\tilde{a}_{k},\forall k\leq k_{0};\mu_{N}^{\tensor\infty}(R_{12}<-\epsilon)<\epsilon/3)\geq3/4.
\end{equation}
Note that
\[
\{Y_{k}^{N,n_1}>\tilde{a}_{k},\forall k\leq k_{0}\}\subset E_{N}^{n_1}.
\]
where $E^{n_1}_N$ is the $E^N$ corresponding to the case where the admissible sequence is $q^{n_1}$ from \prettyref{thm:main-result-1}. (We remind the reader that by definition $(E^N)^c\subseteq\{Y_i=\emptyset\} $
By definition of the latter set and \prettyref{eq:prob-34}, we then know that with
probability at least $3/4$, for $l\neq k\leq k_{0}$, 
\begin{align*}
\left\langle \abs{R_{12}}\indicator{\sigma^{1}\in\tilde{Y}_{k}^{N,n_1},\sigma^{2}\in\tilde{Y}_{l}^{N,n_1}}\right\rangle  & \leq\left\langle \indicator{R_{12}<-\frac{\eps}{3}}\right\rangle +\frac{\epsilon}{3}\left\langle \indicator{\abs{R_{12}}<\frac{\eps}{3}}\indicator{\sigma^{1}\in\tilde{Y}_{k}^{N,n_1},\sigma^{2}\in\tilde{Y}_{l}^{N,n_1}}\right\rangle +\left\langle \indicator{R_{12}>\frac{\epsilon}{3}}\indicator{\sigma^{1}\in\tilde{Y}_{k}^{N,n_1},\sigma^{2}\in\tilde{Y}_{l}^{N,n_1}}\right\rangle \\
 & =I+II+III.
\end{align*}
Notice that 
$I<\frac{\epsilon}{3}$
by choice of $N_{*}$. Similarly by definition of $E_{N}^{n_1}$,
choice of $N_{*}$, and the fact that $q^{n_1} + a_N(n_1)< \eps/3$,
\[
III\leq g_{k,l,\eps_{n_1(N)}}(\bsig)\leq b_{N}(n_1)\leq\frac{\epsilon}{3}
\]
Finally $II<\frac{\epsilon}{3}$,
by definition so that 
\[
\left\langle \abs{R_{12}}\indicator{\sigma^{1}\in\tilde{Y}_{k}^{N,n_1},\sigma^{2}\in\tilde{Y}_{l}^{N,n_1}}\right\rangle \leq\epsilon
\]
as desired.
\end{proof}

 \section{The proof of \prettyref{cor:tal-pure}}\label{sec:tal-pure}
As the proof of \prettyref{cor:tal-pure} a tedious but straightforward modification of the proof of \prettyref{thm:main-result-1}.
In the interest of space, we leave out technical details that are just repetitions up of the arguments above up to a 
small modification. We begin with the following lemmas. Their proofs are exactly as before, so we omit them.
 \begin{lem}
 Let $\eps_n=\frac{1}{2^n}$ and $\Delta>0$.  For all $\Delta$, there is a sequence $N_2(n;\Delta)$
 such that for all $N\geq N_2$,
 \[
 \E \mu_N^{\tensor 2}(R_{12}>q_*+\Delta) < \eps_n
 \]
 \end{lem}
\begin{lem}
For all $\eps,\Delta>0$ and $N$, let 
\[
h^N_{\eps,\Delta}(\sigma)=\mu_N^{\tensor 2}(\sigma^1,\sigma^2:(\sigma^i,\sigma)\geq q_* -\Delta-\eps,\forall i\in[2];
R_{12}\geq q_*+\Delta).
\] 
Fix $M,m$ and let $\eps_n$ be as above. Then for  $N\geq N_2(n;\Delta)$, 
\[
\E\mu^{\tensor\infty}_N\left(\sum_{\alpha\in\partial\tau^i_m} h^N_{\eps_n,\Delta}(\sigma^\alpha)<\sqrt{\eps_n},\forall i\in M\right)
\geq 1-Mm^{r+1}\sqrt{\eps_n}
\] 
\end{lem}
\begin{lem}
Let $\eps_n=1/2^n$. For all $n,\Delta>0$, there is an $N_3(n,\Delta)$ such that for all $N\geq N_3(n;\Delta)$
\[
\zeta_N[q_* - \Delta+\eps_n,1]\geq \zeta\{q_*\} -\eps_n
\] 
\end{lem}

We then have the following modification of \prettyref{lem:clust-1}.

\begin{lem}\label{lem:clust-1-mod}
$\forall\eta,\epsilon,\delta$, and for all $\Delta>0$, if we let $q_r=q_*-\Delta$ and chose $\{q_k\}_{k=1}^{r-1}$ so that the sequence $\{q_k\}_{k=1} ^r$ is $\zeta$-admissible, then there is an $m(\eta,\eps,\delta,\Delta),M(\eta,\eps,\delta,\Delta),n_{0}(\eta,\eps,\delta,\Delta)$ and a sequence $\tilde{N}_1(n;\eta,\eps,\delta,\Delta)$ such that for
$n\geq n_{0}$ and $N\geq \tilde{N}_1(n)$, 
\begin{align*}
\E\mu^{\tensor\infty}_{N}\Bigg(\left(\bigcup_{i=1}^{M}E^\Delta_{\tau^i_{m},\eps,\delta}(\mu_N)\right)&\bigcap_{i=1}^{M}
\left\{ \sum_{\alpha\in\tau_m^i} f_{\abs{\alpha},\epsilon_{n}}(\sigma_N^{\alpha})\leq\sqrt{\eps_n}\right\} \cap
\left\{ \sum_{\alpha\nsim\beta\in\tau_m^i} g^N_{\alpha,\beta,\eps_{n}}(\bsig_N)\leq\sqrt{\eps_n}\right\}\\
&\cap\left\{\sum_{\alpha\in\partial\tau_m^i}h^N_{\eps_n,\Delta}(\sigma^\alpha)\leq \sqrt{\eps_n}\right\} \Bigg)
\geq1-\eta
\end{align*}
where by $E^\Delta_{\tau,\eps,\delta}$ we are making the dependence of $E_{\tau,\eps,\delta}$ on $q_r=q_*-\Delta$ explicit
\end{lem}
\begin{proof}
This is the same as in \prettyref{lem:clust-1}, except now $N_0$ and $N_1$ depend on $\Delta$ as well and we take 
\[
\tilde{N}_1(n;\eta,\eps,\delta,\Delta)=\max\{N_0(\eta,\eps,\delta,\Delta),N_1(n,\Delta),N_2(n;\Delta),N_3(n;\Delta)\},
\]
the expression \prettyref{eq:clust-1-abc} becomes
\[
\E \mu^{\tensor\infty}_N(A\cap B\cap C\cap D) \geq 1-\eta/2-M(m^{2r+2}+m^{r+1}+m^r)\sqrt{\epsilon_n}.
\]
and $n_0$ becomes
$n_0 \geq \lceil2\lg\left(2M(m^{2r+2}+rm^r+m^r)/\eta\right)\rceil.$
\end{proof}
We then have the following proposition.
\begin{prop}\label{prop:clust-hier-mod}
Fix $\{q_k\}_{k\leq r-1}$ and $\Delta_n>0$ such that $\Delta_n\rightarrow0$; such that if $q_r =q_*-\Delta_n$, 
$\{q_k\}_{k\leq r}$ is a $\zeta$-admissible sequence; and such that 
such that that $q_{r-1}<q_*-\tilde{\Delta}_1$. Then there is a monotone increasing function $n_0(N)$ such that
and $\lim_{N\rightarrow\infty}n_0(N)=\infty$ and  a sequence of sets
$\{\tilde{X}_{\alpha,N}\}_{\alpha\in\cA_r}$ such that they are $(\eps_{n_0(N)},0)$-hierarchically exhausting and 
$(\sqrt{\eps_{n_0(N)}},\sqrt{\eps_{n_0(N)}})$-hierarchically clustering for  $\mu_N$,except with the modification that
\prettyref{eq:unif-close} for $k=r$ has $q_r(N)=q_*-\Delta_{n_0(N)}$. Furthermore, these sets have
the property that there are sequences $c_N,d_N,p_N$ tending to zero such that for all $\alpha\in\partial\tau_{m_N}$, 
\begin{equation}\label{eq:L_1-over-ctrl}
\int_{\tilde{X}_{\alpha,N}^2} \abs{R_{12}-q_*}d\mu_N^{\tensor 2} <c_N \mu_N(\tilde{X}_{\alpha,N})^2 +d_N.
\end{equation}
with probability at least $1-p_N$
\end{prop}
\begin{proof}
We begin as in \prettyref{thm:main-result-1-proof} by fixing $\nu\in\N$, and letting $\eta_\nu=\eps_\nu=\delta_\nu=1/2^\nu$.
Consider the set
\begin{align*}
E^N_{\nu}=\bigcup_{i=1}^{M_{\nu}}E^{\Delta_\nu}_{\tau_{m_{\nu}}^{i},\epsilon_{\nu},\delta_{\nu}}(\mu_N)&\cap
\left\{ \sum_{\alpha\in\tau_{m_{\nu}}}f^N_{\abs{\alpha},\eps_{n}}\left(\sigma_N^{\alpha}\right)\leq\sqrt{\eps_{n}}\right\} 
\cap\left\{ \sum_{\alpha\nsim\beta}g^N_{\alpha,\beta,\eps_{n}}(\bsig_N)\leq\sqrt{\epsilon_{n}}\right\}\\
&\cap\left\{\sum_{\alpha\in\partial\tau_{m_\nu}}h^N_{\eps_n,\Delta_\nu}(\sigma^\alpha)\leq \eps_n\right\} 
\end{align*}
Then, by \prettyref{lem:clust-1-mod}, we know that there are $m,M,n_0$ all functions of $\nu$ 
and a sequence $\tilde{N}_1(n;\nu)$ such that for $n\geq n_0$ and $N\geq\tilde{N}_1(n;\nu)$, we have that
\[
\E\mu_N^{\tensor\infty}(E_\nu^N) \geq 1-\frac{1}{2^\nu}
\]
as before. Choose $m,M$, and $n_0$ as before so that they tend to infinity with $\nu$, choose $N(\nu)$ as before,
and define $\nu(N)$ as before. We then define $E^N=E^N_{\nu(N)}$ as before. Then we know that
\[
\E\mu_N^{\tensor\infty}(E^N)\geq 1-\frac{1}{2^{\nu(N)}} 
\] 
as before. Define $i$ and $C_\alpha$ as per \prettyref{thm:main-result-1-proof} and define $\tilde{X}_{\alpha,N}$ and
$X_{\alpha,N}$ in the same way as the $\tilde{Y}_{\alpha,N}$ and $Y_{\alpha,N}$ from \prettyref{sec:clust-reg}.
That these sets have the clustering and exhausting properties as before is the same as in \prettyref{thm:main-result-1}. It
remains to verify \prettyref{eq:L_1-over-ctrl}. It suffices to check this inequality on $E^N$. In this case, we see that
for $\alpha\in\partial\tau_{m_{\nu(N)}}$, 
\begin{align*}
\int_{\tilde{X}_{\alpha,N}^2}\abs{R_{12}-q_*} d\mu_N^{\tensor 2}
&\leq \int_{\tilde{X}_{\alpha,N}^2}\abs{R_{12}-q_*}\Bigg(\indicator{R_{12}>q_*+\Delta_{\nu(N)}}\\
&+\indicator{R_{12}\in [q_*-\Delta_{\nu(N)}-\eps_{n_0(N)},q_*+\Delta_{\nu(N)})} 
+\indicator{R_{12} <q_*-\Delta_{\nu(N)}+\eps_{n_0(N)}} \Bigg)d\mu_N^{\tensor 2}\\
&\leq 2h^N(\sigma^\alpha)+(\Delta_{\nu(N)}+\eps_{n_0(N)})\mu_N(\tilde{X}_{\alpha,N})^2 +2f^N_{r,\eps_{n_0(N)}}(\sigma^\alpha)\\
&\leq 4\sqrt{\eps_{n_0}(N)}+(\Delta_{\nu(N)}+\eps_{n_0(N)})\mu_N(\tilde{X}_{\alpha,N})^2 
\end{align*}
where the last inequality follows from the definition of $\tilde{X}_{\alpha,N}$. Set $c_N=(\Delta_{\nu(N)}+\eps_{n_0(N)})$, $d_N= 4\sqrt{\eps_{n_0}(N)}$ and $p_N=\frac{1}{2^{\nu(N)}}$ to get the result.
\end{proof}
\begin{prop}
If we let $X_{\alpha,N}$ be the weights of the $\tilde{X}_{\alpha,N}$ in \prettyref{prop:clust-hier-mod}, then if we
let $\mathbf{v}^N=(v^N_n)$ be the masses of the leaves arranged in decreasing order, then 
$\mathbf{v}^N\convdist\mathbf{v}$
where $\mathbf{v}$ are the points of a $PD(1-\zeta\{q_*\})$.Furthermore, if we let $\mathbf{w}^N$ be the weights 
$X_\alpha$ placed in standard order, then
$\mathbf{w}^N\convdist\mathbf{w}$
where $\mathbf{w}$ is an $RPC(\Gamma_*\zeta)$ where $\Gamma_*\zeta$ is the push forward of $\zeta$ through $\Gamma$ as defined in \prettyref{eq:Gamma}
\end{prop}
\begin{proof}
This is essentially as before with the following minor modifications. First, modify $\phi_\kappa$ from 
\prettyref{lem:indicator-approximation} to be $\phi_{\kappa,\lambda}$ with $\kappa>\lambda>0$, be the piecewise 
linear function that is $0$ until $q_*-\kappa$ and $1$ after $q_*-\lambda$. Most of the proof of \prettyref{lem:indicator-approximation}
stays the same, replacing $q_r$ with $q_*-\lambda$ and $q_r-\kappa$ with $q_*-\kappa$,except: in $V$ you only have case $(i)$ with $q_*-\Delta_n-\epsilon_{n_0}$, and in $VI$, instead of cutting off the second
interval at $q_*+\kappa$, cut it off at $q_*-\Delta_n+\eps_{n_0}$ and use the bound
\[
\limsup\zeta_N[q_*-\kappa,q_*-\Delta_n+\eps_{n_0})\leq \limsup (\zeta_N[q_*-\kappa,1]-\zeta[q_*-\Delta_m+\eps_{n_0}))\leq \zeta[q_*-\kappa,1)-\zeta\{q_*\}
\]
then all of the proofs in \prettyref{sec:clust-reg} follow through \emph{mutatus mutandis}. 

Similarly for
 \prettyref{lem:indicator-approximation-2}, the arguments do not change if $k<r-1$, if $k=r$ then
$\indicator{A}=U_{12}$ so you can use the above approximation. if $k=r-1$, you let $q_{k+1}$ be
 $q_* - \lambda$,  $II(i)$ is instead naively bounded by
 \[
 \zeta_N[q_*-\lambda,1]
 \] 
 and the second term in $III$ (v) can be ignored.
 Take $\lambda$ to zero first and then $\kappa$ to zero in the approximation \prettyref{cor:rpc-set-approx}
\end{proof}
\begin{cor}
Let $X^\prime_{k,N}$ be the weights $X_{\alpha,N}$ for $\abs{\alpha}=r$ arranged in decreasing order and let 
$\tilde{X}^\prime_{k,N}$ be the corresponding sets. For every $\eps$ positive and $k_0\in\N$ there is an 
$N_0$ such that for $N\geq N_0$, with probability at least $1-\eps$
\[
\int_{\tilde{X}_{k,N}^2} \abs{R_{12}-q_*}d\mu_N^{\tensor 2} <\eps\mu_N(\tilde{X}_{\alpha,N})^2.
\]
for all $k\leq k_0$.
\end{cor}
\begin{proof}
Choose $(a_k)$ as in \prettyref{lem:ak-for-pd} such that if $(v_k)\sim PD(1-\zeta\{q_*\})$, 
\[
\prob(v_k\geq a_k\forall k\leq k_0)\geq 1-\eps/4.
\]
By the weak convergence of the weights $X^\prime_{k,N}$ , we can then choose an $N_1$ such that for $N\geq N_1$
\[
\prob(X^\prime_{k,N}\geq a_k\forall k\leq k_0)\geq 1-\eps/2.
\]
Finally choose $N_2 \geq N_1$ large enough that
$c_N<\eps$, $d_N \leq \eps (a_1\vee\ldots\vee a_{k_0})$ and $\frac{1}{2^{\nu(N)}}<\eps/2$. 
Then by \prettyref{prop:clust-hier-mod} we see that by inclusion-exclusion arguments,
\[
\int_{\tilde{X}_{\alpha,N}^2} \abs{R_{12}-q_*}d\mu_N^{\tensor 2} <\eps \mu_N(\tilde{X}_{\alpha,N})^2
\]
with probability $1-\eps$.
\end{proof}
Putting these together gives us \prettyref{cor:tal-pure}.

\section{Quantitfication of the above results}\label{sec:quant-version}
In this section we discuss how to quantify the rates of convergence in \prettyref{thm:main-result-1-proof}. 
Notice from the proof of \prettyref{thm:main-result-1-proof} that it suffices to compute $N_0$ and $N_1$, and to find
expresions for $m$ and $M$. 

We begin first by proving a quantitative version of \prettyref{prop:regularity}.  This will follow after proving that the
Poisson Point processes involved in the construction of RPC's and the Poisson-Dirichlet process can be localized about 
particular point sets. We combine these two localization results to find $m$ and $M$ which we will denote by $m_*$ and 
$M_*$ respectively. We then use a polynomial approximation argument to obtain the rate of convergence 
in \prettyref{prop:regularity}. This gives us our $N_0$. To get $N_1$ simply set $N_1 = D_2^{-1}(\eps_n)$. The result
then follows immediately from the arguments in \prettyref{thm:main-result-1-proof}.

\subsection{Localization of Poisson-Dirichlet processes}
We begin by recording the following concentration results regarding the Poisson-Dirichlet process. We end by proving 
a quantitative version of the existence of $m$.

Let $(v_n)$ be distributed like a $PD(\theta)$ and let $(X_n)$ be the points of a homogenous Poisson point process
on the half-line $PPP(dx\indicator{[0,\infty)})$ ranked in increasing order.  Recall from \cite{PitYor97} 
that there is an $L$ such that $\Gamma(1-\theta)L$ 
has Mittag-Leffler$(\theta)$ distribution, and such that
that 
\begin{equation}\label{eq:hom-poiss-pd}
X_n \eqdist \frac{L}{v_n^\theta}.
\end{equation}
%With this in hand we then have the following concentration result which tells us that the Poisson-Dirichlet process
%is localized around $L/n^{1/\theta}$. 

With this in hand, it immediately follows from Chernoff's inequality that the Poisson-Dirichlet Process
is localized around $L/n^{1/\theta}$. We summarize this in the following lemma
\begin{lem}\label{lem:pd-conc}
For $0<\delta<1$, 
\[
P\left(v_n \geq \left(\frac{1+\delta}{n} L\right)^{1/\theta}\right)\leq e^{-n\frac{\delta^2}{8}}
\]
\end{lem}

Let
$C(\theta) = \E L^{\frac{1}{\theta}} =\frac{\Gamma(1+\frac{1}{\theta})}{\Gamma(1-\theta)^{1/\theta}}$
where the second equality can be found in \cite{PitYor97}.
We then have by a standard truncation argument:
\begin{lem}\label{lem:pd-int-bd}
For $0<\delta<1$, 
\[
\E v_n \leq e^{-n\delta^2/8}+C(\theta)(1+\delta)^{1/\theta}\frac{1}{n^{\frac{1}{\theta}}}
\]
\end{lem}

We finally make a conclusion about tail sums. Let
\[
\Phi(m;\theta) = \frac{C(\theta)2^{\frac{1}{\theta}+1}}{m^{\frac{1-\theta}{\theta}}} \frac{\theta}{1-\theta} + \frac{e^{-m/8}}{1-e^{-1/8}}.
\]

\begin{lem}\label{lem:pd-sum-tail-bound}
For every $m\geq 2$ natural number,
\[
\E \sum_{n\geq m} v_n \leq \Phi(m;\theta)
\]
\end{lem}
\begin{proof}
Note that by \prettyref{lem:pd-int-bd}, for every $0<\delta<1$, it follows that 
\[
\E \sum_{n\geq m} v_n \leq \sum_{m\leq n} C(\theta)(1+\delta)^{1/\theta} \frac{1}{n^{1/\theta}} + e^{-n\frac{\delta^2}{8}}
\leq \frac{C(\theta)2(1+\delta)^{1/\theta}}{m^{\frac{1-\theta}{\theta}}} \frac{\theta}{1-\theta} +  e^{-m\delta^2/8}\frac{1}{1-e^{-\delta^2/8}}.
\]
Sending $\delta\rightarrow 1$ gives the result. 
\end{proof}

\subsection{Localization of a certain Poisson Point Process}

In this section we show that the $PPP(d\mu_\theta)$ is localized. 
\begin{lem}
Let $\eta\in(0,1)$ and $m$ a natural number. Let
$b = (\log(1/\eta)+4m+1)^{1/\theta} = b(\eta,\eps;\theta)$
Then 
\[
P(N(1/b,b)\leq m)\leq \eta
\]
\end{lem}
\begin{proof}
Let $\lambda = \E N(1/b,b)$, $\xi = e^{-3m}\eta$, and $\delta = m/\lambda$. Since $b\geq 1$, it
follows that 
\[
\lambda =\int_{1/b}^b d(-y^{-\theta}) \geq b^{\theta}-1 = \log(1/\eta)+4m = \log(1/\xi)+m
\]
by definition of $b$. This implies that
\[
\delta \leq \frac{m}{\log(1/\xi)+m}\leq 1/4.
\]
Chernoff's inequality then implies that %%%MARK
\[
P(N(1/b,b)\leq m)=\frac{e^{-\lambda}(e\lambda)^m}{m^m} =e^{-m g(\delta)}
\]
for $\delta \leq 1/4$, where $g(\delta) = \frac12 (\frac{(1-\delta)}\delta)$. 
Since $g$ is strictly decreasing, we see that  since $\delta\leq \delta^\prime = \frac{m}{\log(1/\xi)+m}$, 
it follows that
\[
g(\delta)\geq g(\delta^\prime) = \log(1/\xi)/m \text{ so that }
P(N(\frac1b ,b)\leq m)\leq \eta
\]
\end{proof}

We record the following useful corollary.
\begin{cor}
Fix sequence $0<\zeta_1<\ldots<\zeta_r<1$, an $\eps,\eta$ and an $m$, and let
\[
\bar{b}(\eta,\eps,m;\{\zeta_k\}) = b(\frac{\eta}{m^r},\eps,\zeta_1).
\]
For each $\alpha\in\tau_m\setminus\partial\tau_m$ associate an independent $\Pi_\alpha = PPP(d\mu_{\zeta_\abs{\alpha}})$. Then
\[
P(\exists \alpha\in\tau_m: N_\alpha(\frac1b , b)\leq m)\leq \eta
\]
\end{cor}
\subsection{Quantification of \prettyref{lem:exhausting-tau}}
We begin first by quantifying the existence of $m$. In the following, we remind the reader that 
$\zeta_k = \zeta[0,q_{k+1})$.

Let $m_k(\eps,\eta,r;\zeta)$ be defined iteratively as follows. Let $m_0$ solve
\[
\Phi(m_0,\zeta_0)=\frac{\eta\eps}{r^2}.
\]
let $\bar{m}_k = \prod_{l\leq k-1} m_l$. Next let  $m_k$ solve
\[
\Phi(m_k,\zeta_k) = \frac{\eta\eps}{(r\bar{m}_k)^2}
\]
Finally, let $m_*(\eps,\eta,\zeta) = \max_k m_k$. This will be the relevant $m$.  Finally
for convenience we define
\[
\bar{A}_{\eps,m} = \bigcap_{k\leq r}\{ \sum_{\abs{\alpha}=k,\alpha\in\tau_m} v_\alpha>1-\eps\}
\bigcap_{\alpha\in\tau_m\setminus\partial\tau_m}\{ v_\alpha -\sum_{\beta\in child(\alpha)} v_\beta\in (0,\eps)\}
\]
\begin{lem} \label{lem:mstar}
For $\eps,\eta$ positive, we have that
\[
Q(\cup_{i=1}^\infty A^i_{\tau_{m_*},\eps,\delta})\geq \prob(\bar{A}_{\eps,m_*}) \geq 1-\eta
\]
\end{lem}
\begin{proof}
The first inequality is clear by the same argument as in \prettyref{lem:q-a-lb}.  To get the second inequality 
we proceed as follows. By \prettyref{lem:pd-sum-tail-bound}, Markov's inequality, and choice of $m_0$
\[
P(\sum_{n\leq m_0} v_n \leq \eps/r)\leq \frac{r \Phi(m_0;\zeta_1)}{\eps}\leq \eta/r.
\]
Take $k\leq r$, and $\alpha\in\tau=\tau(m_0,\ldots,m_{r-1})$ with  $\abs{\alpha}=k$.
Since 
\[
v_{\alpha n} \eqdist \frac{w_\alpha u_{\alpha n}}{\sum w_\alpha \sum u_{\alpha n}} \leq \frac{ u_{\alpha n}}{\sum_n u_{\alpha n}} \eqdist \tilde{v}_k^{\abs{\alpha}}
\]
where $(\tilde{v}^k_n)\sim PD(\zeta_{k-1})$. It follows that if $\abs{\alpha}=k$, 
\[
P(\sum_{m_k\leq n} v_{\alpha n} \geq \frac{\eps}{r \bar{m}_k} )\leq P(\sum_{m_k\leq n}  \tilde{v}_n^k \geq \frac{\eps}{r\bar{m}_k})\leq \frac{\eta}{r\bar{m}_k}
\]
again by choice of $m_k$. Combining these, we get
\begin{align*}
P(v_\alpha-\sum_{n\leq m_{\abs{\alpha}}}v_{\alpha n}\in (0,\frac{\eps}{r\bar{m}_k}),\forall \alpha\in \tau,
\sum_{n\leq m_0}v_n>1-\eps) \geq 1-\frac{\eta}{r}(1+m_1 \frac{1}{m_1} +\ldots + \bar{m}_r\frac{1}{\bar{m}_r})\geq 1-\eta
\end{align*}

To get the required lower bound, it suffices to show that the above event is contained
in $\bar{A}_{\eps,m}$. To see this first note that on this event
\[
0< v_\alpha - \sum v_{\alpha n} \leq \frac{\eps}{r \bar{m}_k},
\]
it follows that 
\[
\sum_{\beta\in\tau,\abs{\beta}=k+1}v_\beta=\sum_{\alpha\in\partial\tau(m_1,\ldots,m_{k-1})} \sum_{n\leq m_k} v_{\alpha n} \geq \sum_{\alpha\in\partial \tau(m_1,\ldots, m_k)}(v_\alpha -\frac{\eps}{r \bar{m}_k}) 
= (\sum_{\alpha\in\tau,\abs{\alpha}=k} v_\alpha) - \eps/r.
\]
Iterating this lower bound for decreasing $k$ (i.e. up the tree) gives
\[
\sum_{\beta\in\tau,\abs{\beta}=k+1}v_\beta \geq 1-\frac{\eps}{r} r.
\]
Finally notice that
\[
\sum_{\alpha\in\partial \tau(m_1,\ldots,m_k)} v_\alpha \leq \sum_{\abs{\alpha}=k,\alpha\in\tau_{m_*}}v_\alpha 
\]
So that if $\alpha\in \tau_{m_*}$ with $\abs{\alpha}=k$ and $\alpha_k \geq m_k$,  it follows that
\[
v_\alpha -\sum_{n\leq m_*} v_{\alpha n} <v_\alpha <\eps.
\]
Thus we have the set containment we desire.
\end{proof}

Before we state the main result of this section, we need a few more definitions. Let
\[
p_*(\eps,\eta)= (m_*(\eps,\frac{\eta}{4}) \bar{b}(\eps,\eta/4)^2)^{-rm^r}
\text{ and }
M_*(\eta,\eps)=\frac{\log(\eta)}{\log(1-p_*)}.
\]
With these, we then have
\begin{prop}[Quantification of \prettyref{lem:exhausting-tau} ]
Fix $\eps$ and $\eta$ positive. Then 
\[
Q(\cup_{i=1}^{M_*} A_{\eps,\delta}^{\tau_{m_*}^i}) \geq 1-\eta
\]
\end{prop}
\begin{proof}
Recall by \prettyref{lem:mstar}, it follows that
\[
\prob(\bar{A}_{\eps,m_*(\eps,\eta/4)})\geq 1-\frac\eta4
\]
so that
\[
\prob(\bar{A}_{\eps,m_*} \cap\{\forall \alpha\in\tau_{m_*}: N_\alpha(\frac1b ,b) \geq m_*\}) \geq 1-\frac\eta2.
\]
Call this set $B$ for this lemma. Notice that if  $(u_{\alpha n})$ are the ranked elements of $\Pi_\alpha$ and are independent over $\alpha$, then for $\alpha\in\partial \tau_m$, 
\[
v_\alpha = \frac{\prod_{\beta\precsim \alpha} u_\beta}{\sum_{\alpha\in\partial\tau_m*} \prod_{\beta\precsim\alpha} u_\beta}
\geq (m_*\bar{b}^2)^{-r}
\]
Note that if $\alpha$ is not a leaf of $\tau_m$, then there is some leaf that lower bounds it since $v_\beta =\sum_{\beta\precsim \alpha}v_\alpha$. Thus this bound extends to the whole tree. It then follows that
\[
\prod_{\alpha\in\tau_{m_*}} v_\alpha \geq (m \bar{b}^2)^{-rm^r} =p_*.
\]

Let $X_i =\indicator{A_{\eps,\delta}^{\tau^i}}$, and let $I=\min\{i:X_i = 1\}$ be the first time this sequence is 1. Then 
conditionally on $(v_\alpha)$, it follows that $I$ is a geometric random variable with parameter $p$ where $p\geq p_*$ on $B$. Notice then that
\[
\prob(I\geq M_*,B) \leq (1-p_*)^{M_*} \leq \eta/2
\]
Combining these results then gives
\[
Q(\cup_{i=1}^{M_*} A_{\eps,\delta}^{\tau^i_{m_*}}) = \prob(I\leq M_*) \geq P(I\leq M_*,B) \geq\prob(B)-\prob(I\geq M_*,B) \geq 1-\eta/2 - \eta/2 
\]
\end{proof}

\subsection{An approximation theorem}
For $(x_1,\ldots,x_d)\in[0,1]^d$, let $X_n(x_1,\ldots,x_d)$ be the random $d$-vector whose
$i$-th entry is an independent Binomial$(n,x_i)$ random variable. For $f\in C([0,1]^d)$, define the 
operator 
\[
B_n f(x) = \E f(\frac{X_n(x_1,\ldots,x_d)}{n})= \sum_{k_1,\ldots,k_d\leq n} \sum_{l_1\leq n-k_1}\ldots\sum_{l_d\leq n-k_d} f(\frac{k_1}{n},\ldots,\frac{k_d}{n})
	\prod_{i=1}^d \binom{n}{k_i,l_i} (-1)^{k_i+l_i}\E(X_n^{(i)})^{k_i+l_i}
\]
The following is a basic consequence of Chebyshev's inequality.

\begin{lem}
if $f\in C([0,1]^d)$ is Lipschitz in the $\ell_1$ norm, then 
\[
\norm{f-B_n f(x)}_{\infty} \leq \frac{d}{2\sqrt{n}} \norm{f}_{lip}
\]
\end{lem}

In the following, let $\mathscr{A}$ is the class of $A$ as in \prettyref{eq:F2-cty-set}. 
\begin{lem}\label{lem:poly-approx}
Let $f$ be lipschitz on $([0,1]^d,\norm{\cdot}_1)$ where $d=M\abs{\tau_m}((\abs{\tau_m}-1)/2 + 1)$, and let $\bar{Q}$ and $\bar{Q}_N$ the the restriction
of $Q$ and respectively $Q_N$ to the coordinates indexed by $\cup_{i=1}^M \tau^i_m\subset\cA_r$.
\[
\abs{\int B_n f(x) d\bar{Q}(x) - \int B_n f(x)d\bar{Q}_N(x)}\leq n^2 2^{2nd-1} \norm{f}_\infty \sup_{A\in\mathscr{A}} \abs{P_N(A)-P(A)}
\]
\end{lem}
\begin{proof}
To condense notation, let $\mathbf{k}=(k_1,\ldots,k_d)$ and similarly for $\mathbf{l}$. All vector inequalities  are to be interpreted coordinate-wise. 
We begin by noting that
\begin{align*}
\abs{\int B_n f(x) d\bar{Q}(x) - \int B_n f(x) d\bar{Q}_N(x)} &= \abs{\sum_{\mathbf{k}\leq n}\sum_{\mathbf{k}+\mathbf{l}\leq n} 
f(\frac{\mathbf{k}}{n}) \prod_{i=1}^d \binom{n}{k_i,l_i} \left(\int x_j^{k_j+l_j} d(\bar{Q}-\bar{Q}_N)\right)}\\
&\leq  \sum_{\mathbf{k}\leq n}\sum_{\mathbf{k}+\mathbf{l}\leq n} \norm{f}_\infty \left(\prod_{i=1}^d\binom{n}{k_i,l_i}\right)
\abs{\int x_j^{k_j+l_j} d(\bar{Q}-\bar{Q}_N)}.
\end{align*}
Recall that there is an $A(\mathbf{k},\mathbf{l})$ in $\mathscr{A}$ corresponding to the above monomials, as explained in \prettyref{eq:F2-cty-set}. Keeping this in mind, we then see that
\begin{align*}
RHS&\leq \sum_{\mathbf{k}\leq n}\sum_{\mathbf{l}+\mathbf{k}\leq n} \norm{f}_\infty (\prod_i \binom{n}{k_i,l_i} \abs{P_N(A(\mathbf{k},\mathbf{l}))-P(A(\mathbf{k},\mathbf{l}))}\\
&\leq \#\{ \mathbf{k}\leq n,\mathbf{l}+\mathbf{k}\leq n\} \norm{f}_\infty \binom{n}{n/2}^{2d} \sup_{A\in\mathscr{A}}\abs{P_N(A)-P(A)}
\leq (\frac{n^2}{2})2^{2nd} \norm{f}_\infty \sup_{A\in\mathscr{A}}\abs{P_N(A)-P(A)}
\end{align*}
\end{proof}

\subsection{A Quantified version of \prettyref{prop:regularity}}
Let $A_{\eps,\delta,m,M} = \cup_{i=1}^M A_{\tau_m,\eps,\delta}^i$, let 
\begin{equation}
\bar{A}_{\tau,\epsilon,\delta} 
=\bigcap_{k=1}^{r}\left\{\sum_{\substack{\abs{\alpha}=k\\\alpha\in\tau}}x_{\{\alpha\}}>1-\epsilon\right\}\bigcap_{\alpha\in\tau\setminus\partial\tau}\left\{\abs{x_{\{\alpha\}}-\sum_{\beta\in child(\alpha)}x_{\{\beta\}}}<\eps\right\}
\bigcap\left\{\abs{\sum_{\substack{\alpha\nsim\beta\\\alpha\beta\in\tau}}x_{\alpha,\beta}} <\delta/\abs{\tau}^{2}\right\},
\end{equation}
and let 
\[
\bar{A}_{\eps,\delta,m,M}=\cup_{i=1}^M \bar{A}_{\tau_m,\eps,\delta}^i
\]
\begin{lem}[\prettyref{prop:regularity} quantified]\label{lem:reg-quant}
For every $\eps,\delta,\eta$ positive with $\eps$ and $\delta$ smaller than 1, let $m=m_*(\eps/2,\delta/2)$ and 
$M=M_*(\eta,\eps/2)$ , $d$ be as in
\prettyref{lem:poly-approx}, and let 
$n= \left(\frac{16 d \abs{\tau_m}}{\eta\eps\delta}\right)^2$.
For $N \geq D^{-1}(\eta/(n^2 2^{2nd+1}))=:N_0$,
\[
Q_N(\bar{A}_{\eps,\delta,m,M})\geq 1-\eta
\]
\end{lem}
\begin{proof}
Let 
\[
\psi_{\iota} = \psi_{\iota}^{\eps,\delta} = 1-\frac{d_{\ell_1}(x,A_{\frac{\eps}2,\frac\delta 2})\wedge \iota}{\iota}
\]
and set $\iota = \frac{\eps\delta}{4\abs{\tau_m}}$. 
Note that $\norm{\psi_\iota}_\infty \leq 1$ and $\norm{\psi}_{lip} \leq 1/{\iota}$.  Since the functions
in the definition of $A_{\eps,\delta,m,M}$ are all $1$-lipschitz on  $(\R^d,\norm{\cdot}_1)$,  it follows that
\[
\indicator{A_{\frac{\eps}2,\frac\delta 2,m,M}}\leq \psi_\iota \leq \indicator{\bar{A}_{\eps,\delta,m,M}}.
\]
Recall that by choice of $m$ and $M$, it follows that 
\[
1-\eta/2 \leq Q(A_{\eps/4,\delta/4,m,M}).
\]
 Then 
\begin{align*}
Q(A_{\eps/2,\delta/2,m,M})&\leq \int\psi_\iota d\bar{Q}\leq \int B_n\psi_\iota(x)d\bar{Q} + \frac{d}{2\sqrt{n}\iota} \leq \int B_n \psi_\iota (x)d\bar{Q}_N + n^2 2^{2nd-1}D(N) + \frac{d}{2\sqrt{n}\iota}\\
&\leq \int \psi_\iota d\bar{Q}_N + \frac{d}{\sqrt{n}\iota} + n^2 2^{2nd - 1} D(N)\leq Q_N(\bar{A}_{\eps,\delta,m,M}) +\frac{d}{\sqrt{n}\iota} + n^2 2^{2nd - 1} D(N)
\end{align*}
so that
\[
Q_N(\bar{A}_{\eps,\delta,m,M})\geq 1-\frac{\eta}{2}-\left(\frac{d}{\sqrt{n}\iota} + n^2 2^{2nd - 1} D(N)\right)
\]
By choice of n, 
\[
\frac{d}{\sqrt{n}\iota}\leq \eta/4
\text{ and similarly by choice of $N$, }
D(N)\leq \frac{\eta}{2n^22^{2nd-1}}.
\]
Combining these with the above inequality gives the result.
\end{proof}

\begin{rem}
Notice that $\bar{A}$ has an absolute value rather than a lower 
bound of $0$ for certain inequalities. Following through the proof 
of \prettyref{prop:exhaustion-1}, one sees that one can still 
create an $(\eps,0)$-exhaustion as desired.
\end{rem}
\subsection{Lower bounds on rates: proof of \prettyref{thm:quant-clust}}
One might be further interested in lower bounds on rates
in approximate ultrametricity. 
Before we begin we make a few simplifications to make the analysis easier:
Let $\tilde{K}(\theta)=4C(\theta)\theta/(1-\theta)+10$. Let $K$ be such that 
\[
K=\lceil\min_{\theta\in\{\zeta_k\}} \tilde{K}(\theta)\rceil.
\]
Finally let $1/\alpha=\min{(1-\zeta_k)/\zeta_k,1/8}$. Since $e^{-m/8}\leq 1/m^{1/8}$, it follows that
\[
\Phi(m;\theta) \leq \frac{K}{m^{1/\alpha}}.
\]
Notice that in the above all of the arguments still hold if we increase $m_k$.
In particular, by inductive arguments one can show that we can let 
\[
m_*=\left(\frac{Kr^2}{\eta\eps}\right)^{2\alpha(\alpha+1)^{r-1}}
\]

\begin{thm*}[\ref{thm:quant-clust}]
Let $\{q_k\}_{k=1}^r$ be an admissible sequence with parameters
$\zeta_k=\zeta[0,q_k)$. Then there are functions $n_0(N)$
$\nu(N)$, and $m_{**}(N)$ such that $\mu_N$ admits a collection
$\{C_{\alpha,N}\}_{\alpha\in\tau_{m_{**}(N)}}$ that
is $(2^{-\nu+1},0$-hierarchically exhausting and 
$(1/2^{n_0},1/2^{n_0/2})$-hierarchically clustering.
Furthermore we have the bounds,
\begin{align*}
\nu (N)&\geq \Omega(\log\log\log\log(1/D(N)))\\
n_0 (N) &\geq \Omega(\log\log\log\log(1/D(N)))\\
m_{**}(N) &\geq \Omega((\log\log\log(1/D(N)))^c)
\end{align*}
where these inequalities are to be understood 
up to constants that depend on $r$
and $\zeta_1$, and $c$ also depends on these parameters.
\end{thm*}

\begin{proof}
Notice that in the above it suffices to chose an $m_*$ larger
than chosen above, in particular one can chose
\[
m_{**}=\left(Kr^2 2^{2\nu}\right)^{2\alpha(\alpha+1)^{r-1}}
\]
(note that this implies $m_{**}\geq 5\wedge\nu)$.Then by making a bigger choice for $\bar{b}$, we can only 
make the inequalities worse, in particular chose
$\bar{\bar{b}}(\epsilon,\eta/4)=(5m_{**})^{1/\zeta}$
We can then choose $p_*$ and $M_*$ as 
\[
p_{**}=(m_{**}^{\frac{4}{\zeta_1}+1})^{-rm_{**}^r}.
\text{ and }
M_{**}=4\exp\left(r m_{**}^{r+1}(\frac{4}{\zeta_1}+2)\right).
\]
Then $d\leq M_{**}m_{**}^{2r}.$
Since $n^2 2^{2nd+1}\leq 2^{3nd}$ by our choice of $m_{**}$,
We see that  from \prettyref{lem:reg-quant},
$N_0 =D^{-1}(\exp[-I(\nu)])$
where 
\begin{align*}
I(\nu) &= 9(16 m_{**}^{2r} 2^{6\nu})^2 (M_{**} m_{**}^{2r})^3)= C_1 \exp\left(c_2 m_{**}^{r+1}\right)= C_1\exp\left(c_3 e^{c_4 \nu}\right).
\end{align*}
Finally, note that we're free to chose $n_0$ bigger than in
\prettyref{lem:clust-1} so that, in particular we can choose
\begin{align*}
n_0 &=\frac{2}{\log2}[\log(12)+(r(\frac{4}{\zeta_1}+1)m_{**}^{r+1})\leq \tilde{C_1}+\tilde{C_2}2^{c_3\nu}
\end{align*}
Putting these together gives the results
\end{proof}

\section*{Appendix}

\subsection*{Ruelle Probability Cascades}
In this section we define Ruelle Probability Cascades with parameters $0<\zeta_0<\ldots<\zeta_{r-1}<1$ 
and $0=q_0<q_1<\ldots<q_r\leq 1$ and state a few useful properties. When the parameters $\zeta_k$ or $q_k$ satisfy
these conditions we call them \emph{admissible}.
First we begin by defining the weight distribution. (In the above, when
we refer to the weights of an RPC with parameters $\zeta_k$, we mean the weights as we shall define below.
This definition will not depend on the choice of $q_k$, so this is well defined.) Let 
\[
\mu_{\zeta}(dx) = \zeta x^{-(\zeta+1)}dx.
\]
For each $\alpha\in\cA_r\setminus\partial\cA_r$, we associate the weights sequence $(u_{\alpha n})$ which is given
by independently drawn $PPP(\mu_{\zeta_{\abs{\alpha}}})$ arranged in decreasing order. Let 
\[
w_\alpha = \prod_{\beta\precsim\alpha}u_\beta
\text{ and finally let  }
v_\alpha = \frac{w_\alpha}{\sum_{\abs{\beta}=\abs{\alpha}} w_\beta},
\]
This gives us a collection $(v_\alpha)\in\cC_r$ (see \prettyref{defn:cascade-space} for this notation). 
Note that by construction $(v_\alpha)$ is in standard order. 

The corresponding RPC is then defined by
\[
RPC = \cR(\bv,\{q_k\}_{k=1}^r)
\]
where $\mu(\cdot)$ is as in \prettyref{eq:rost-map}. We end this section with the following facts which can
be found, for example, in \cite{PanchSKBook}. Let $\zeta$ be defined by $\zeta\{q_k\} = \zeta_k -\zeta_{k-1}$ where 
$\zeta_{-1}=0$ and $\zeta_r=1$.

\begin{prop}
Fix an RPC with parameters $\{\zeta_k\}$ and $\{q_k\}$. Then 
\begin{enumerate}
\item  For fixed $\alpha$, the weights $u_{\alpha n}$ are strictly positive and distinct almost surely.
\item Fix a $k\leq r$. Then if we let $(\tilde{v}_n)$ be the  weights $(v_\alpha)_{\abs{\alpha}=k}$ arranged in decreasing 
order, $(\tilde{v}_n)\sim PD(\zeta_k)$
\item The RPC satisfies the Ghirlanda-Guerra Identities. 
\item The overlap distribution for the RPC is given by $\zeta$. 
\end{enumerate} 
\end{prop}
Using these facts, we see that one can also uniquely define $RPC(\zeta)$.  We would like to record here a 
consequence of the above ROSt encoding argument. This is to be compared with the similar result for the
Poisson-Dirichlet Process.

\subsection*{A consequence of the Ghirlanda-Guerra identities}

\begin{fact*}[\ref{fact:RPC}]
 Partition the support of $\mu$ in the balls of radii corresponding to
 the overlaps $\{q_{k}\}_{k=1}^{r}$ such that $\zeta[q_{k},q_{k+1})>0$. 
Let $(V_\alpha)_{\alpha\in\cA_r}$ denote the law of the masses of this partition
arranged in standard order. The law of these weights is distributed like those of an
RPC with the overlap distribution with parameters $\zeta_k-\zeta_{k-1}=\zeta[q_k,q_{k+1})$.
In particular, there are infinitely
many of them at each level and they have almost surely non-zero weights. 
\end{fact*}
\begin{proof}

Since the $V_\alpha$ are in standard order, consider
\[
\nu = \cR((V_\alpha),\{q_k\};\cdot)
\]
the ROSt corresponding to $(V_\alpha)$ as per \prettyref{eq:rost-map}. Let $\Gamma$ be as in \prettyref{eq:Gamma}.  

Notice that by definition, if we let consider the $\Gamma$-approximator to $\mu$, call it $\mu^\Gamma$ Notice that 
\[
\E\gibbs{f(R^n)}_{\mu^\Gamma} = \E\gibbs{f((\Gamma(R_{ij})))}_\mu,
\]
so that the $\Gamma$-approximator satisfies the Ghirlanda-Guerra Identities. Note that by definition, if $B_\alpha$ be the balls corresponding to the $V_\alpha$, 
\[
\mu(B_\alpha) = \nu\{h_\alpha\}
\]
so that 
\[
\E\gibbs{f(R^n)}_{\mu^\Gamma}=\E\gibbs{f(R^n)}_{\nu}.
\]
Note that if we set $q_0=0$
\[
\Gamma_*\zeta\{q_k\}=\begin{cases}\zeta[q_k,q_{k+1}) &\text{ for } 0\leq k<r\\
 \zeta[q_r,1] &\text{ for } k=r.\end{cases}
\]
 Let $\eta$ be the random measure corresponding to a Ruelle Probability Cascade with jumps $\zeta_k$. Since RPC's 
satisfy the Ghirlanda-Guerra identities, we know by \prettyref{prop:GGI-cons} that it is uniquely identified by its overlap
distribution, which is the same as that of $\nu$. In particular, we know that there is a coupling of $\eta$ and
$\nu$ and a random isometry $T$ from the closed linear span of the support of $\eta$ to that of $\nu$ such that
$T_*\eta^\Gamma = \nu$. In particular, if we let 
\[
\eta\{Th_\alpha\}=\nu\{h_\beta\}
\]
for some $\beta$. 

To get the equality of the weights, we do the following iterative comparison. Let $\{\tilde{B}_\alpha\}$ be the balls defined
by grouping the $h_{\beta}$ where $\alpha\prec\beta$.  let $v_\alpha$ be the weights corresponding to those balls
from the RPC. Then we see that there is a random bijection $\pi:\cA_r\rightarrow\cA_r$ that preserves the parent child
relationship such that
\[
(v_\alpha)=(V_{\pi(\alpha)})
\]
induced by $T$ since the latter is an isometry and the $\tilde{B}_\alpha$ are a partition. Since the $v_n$ are almost
surely distinct, we see that the $V_n$ must be as well. Consequently, $\pi$ must be the identity map at the top level.
Continue this argument iteratively down the tree. Thus $\pi=Id$.
\end{proof}

\bibliographystyle{plain}
\bibliography{../../../../Summary/spinglass}

\begin{thebibliography}{10}

\bibitem{AldExch83}
David~J. Aldous.
\newblock Exchangeability and related topics.
\newblock In P.L. Hennequin, editor, {\em \'Ecole d'\'Et\'e de Probabilit\'es
  de Saint-Flour XIII Ñ 1983}, volume 1117 of {\em Lecture Notes in
  Mathematics}, pages 1--198. Springer Berlin Heidelberg, 1985.

\bibitem{Arg08}
Louis-Pierre Arguin.
\newblock A remark on the infinite-volume gibbs measures of spin glasses.
\newblock {\em Journal of Mathematical Physics}, 49(12):--, 2008.

\bibitem{ArgZind14}
Louis-Pierre Arguin and Olivier Zindy.
\newblock Poisson-dirichlet statistics for the extremes of a log-correlated
  gaussian field.
\newblock {\em The Annals of Applied Probability}, 24(4):1446--1481, 08 2014.

\bibitem{AufChen13}
Auffinger {Auffinger} and Wei-Kuo {Chen}.
\newblock {On properties of Parisi measures}.
\newblock {\em Probability Theory and Related Fields}, to appear, March 2013.

\bibitem{Bert06}
Jean Bertoin.
\newblock {\em Random fragmentation and coagulation processes}, volume 102 of
  {\em Cambridge Studies in Advanced Mathematics}.
\newblock Cambridge University Press, Cambridge, 2006.

\bibitem{Bov12}
Anton Bovier.
\newblock {\em Statistical Mechanics of Disordered Systems}.
\newblock Cambridge, 2012.

\bibitem{BovKurk04-1}
Anton Bovier and Irina Kurkova.
\newblock Derrida's generalised random energy models. {I}. {M}odels with
  finitely many hierarchies.
\newblock {\em Ann. Inst. H. Poincar\'e Probab. Statist.}, 40(4):439--480,
  2004.

\bibitem{BovKurk04-2}
Anton Bovier and Irina Kurkova.
\newblock Derrida's generalized random energy models. {II}. {M}odels with
  continuous hierarchies.
\newblock {\em Ann. Inst. H. Poincar\'e Probab. Statist.}, 40(4):481--495,
  2004.

\bibitem{ContStarr13}
Pierluigi Contucci, Emanuele Mingione, and Shannon Starr.
\newblock Factorization properties in d-dimensional spin glasses. rigorous
  results and some perspectives.
\newblock {\em Journal of Statistical Physics}, 151(5):809--829, 2013.

\bibitem{Der80}
Bernard Derrida.
\newblock Random-energy model: Limit of a family of disordered models.
\newblock {\em Physical Review Letterss}, 45:79--82, Jul 1980.

\bibitem{Der85}
Bernard Derrida.
\newblock A generalization of the random energy model which includes
  correlations between energies.
\newblock {\em Journal de Physique-Lettres}, 46(9):401--407, 1985.

\bibitem{DFM94}
Vik Dotsenko, Silvio Franz, and Marc M\'ezard.
\newblock Partial annealing and overfrustration in disordered systems.
\newblock {\em Journal of Physics A: Mathematical and General}, 27(7):2351,
  1994.

\bibitem{DymMckean}
Harry Dym and Henry~P. McKean.
\newblock {\em Fourier series and integrals}.
\newblock Academic Press, New York-London, 1972.
\newblock Probability and Mathematical Statistics, No. 14.

\bibitem{Kal97}
Olav Kallenberg.
\newblock {\em Foundations of modern probability}.
\newblock Probability and its Applications (New York). Springer-Verlag, New
  York, 1997.

\bibitem{Mez84}
Marc M\'ezard, Giorgio Parisi, Nicolas Sourlas, G\'erard Toulouse, and Miguel
  Virasoro.
\newblock {Replica Symmetry-Breaking and the Nature of the Spin-Glass Phase}.
\newblock {\em J. Physique}, 45:843, 1984.

\bibitem{MPV87}
Marc M{\'e}zard, Giorgio Parisi, and Miguel~Angel Virasoro.
\newblock {\em Spin glass theory and beyond}, volume~9.
\newblock World scientific Singapore, 1987.

\bibitem{Panch12}
Dmitry Panchenko.
\newblock {The {S}herrington-{K}irkpatrick Model: An Overview}.
\newblock {\em Journal of Statistical Physics}, 149(2):362--383, 2012.

\bibitem{PanchHEpure2013}
Dmitry Panchenko.
\newblock {Hierarchical exchangeability of pure states in mean field spin glass
  models}.
\newblock {\em ArXiv e-prints}, July 2013.

\bibitem{PanchUlt13}
Dmitry Panchenko.
\newblock The {P}arisi ultrametricity conjecture.
\newblock {\em Ann. of Math. (2)}, 177(1):383--393, 2013.

\bibitem{PanchSKBook}
Dmitry Panchenko.
\newblock {\em The Sherrington-Kirkpatrick Model}.
\newblock Springer, 2013.

\bibitem{PanchSGSD13}
Dmitry Panchenko.
\newblock Spin glass models from the point of view of spin distributions.
\newblock {\em The Annals of Probability}, 41(3A):1315--1361, 05 2013.

\bibitem{PitYor97}
Jim Pitman and Marc Yor.
\newblock The two-parameter poisson-dirichlet distribution derived from a
  stable subordinator.
\newblock {\em The Annals of Probability}, 25(2):855--900, 04 1997.

\bibitem{Rue87}
David Ruelle.
\newblock A mathematical reformulation of derrida's rem and grem.
\newblock {\em Communications in Mathematical Physics}, 108(2):225--239, 1987.

\bibitem{TalBK03}
Michel Talagrand.
\newblock {\em Spin glasses: a challenge for mathematicians}, volume~46 of {\em
  Ergebnisse der Mathematik und ihrer Grenzgebiete. 3. Folge. A Series of
  Modern Surveys in Mathematics [Results in Mathematics and Related Areas. 3rd
  Series. A Series of Modern Surveys in Mathematics]}.
\newblock Springer-Verlag, Berlin, 2003.
\newblock Cavity and mean field models.

\bibitem{TalPM06}
Michel Talagrand.
\newblock Parisi measures.
\newblock {\em Journal of Functional Analysis}, 231(2):269 -- 286, 2006.

\bibitem{Tal07}
Michel Talagrand.
\newblock Large deviations, {G}uerra's and {A.S.S.} schemes, and the {P}arisi
  hypothesis.
\newblock {\em Journal of Statistical Physics}, 126(4-5):837--894, 2007.

\bibitem{Tal09}
Michel Talagrand.
\newblock Construction of pure states in mean field models for spin glasses.
\newblock {\em Probability Theory and Related Fields}, 148(3-4):601--643, 2010.

\bibitem{TalBK11}
Michel Talagrand.
\newblock {\em Mean field models for spin glasses. {V}olume {I}}, volume~54 of
  {\em Ergebnisse der Mathematik und ihrer Grenzgebiete. 3. Folge. A Series of
  Modern Surveys in Mathematics [Results in Mathematics and Related Areas. 3rd
  Series. A Series of Modern Surveys in Mathematics]}.
\newblock Springer-Verlag, Berlin, 2011.
\newblock Basic examples.

\end{thebibliography}
\end{document}